\documentclass[a4paper,12pt]{amsart}
\usepackage{amssymb,amsmath}
\let\amslrcorner\lrcorner
\usepackage{mathabx}
\usepackage{stmaryrd}
\usepackage[notref,notcite]{showkeys}
\usepackage{amscd,amsthm,amssymb}
\usepackage{enumerate}
\usepackage{color}
\usepackage[all,cmtip]{xy}
\usepackage{hyperref}
\usepackage{mathrsfs} 

\let\lrcorner\amslrcorner

\usepackage{latexsym}

\usepackage{hyperref}
\hypersetup{
    colorlinks=true, 
    linktoc=all,     
   linkcolor=blue,  
}

\def\sideremark#1{\ifvmode\leavevmode\fi\vadjust{\vbox to0pt{\vss
\hbox to 0pt{\hskip\hsize\hskip1em%
\vbox{\hsize2cm\tiny\raggedright\pretolerance10000%
\noindent {\color{red}{#1}}\hfill}\hss}\vbox to8pt{\vfil}\vss}}}%

\def\.{\cdot}

\def\d{{\mathrm d}}

\def\la{\langle}
\def\ra{\rangle}

\def\t{\tilde}
\def\beq{\begin{equation}}
\def\eeq{\end{equation}}
\def\bea{\begin{eqnarray*}}
\def\eea{\end{eqnarray*}}
\def\beaa{\begin{eqnarray}}
\def\eeaa{\end{eqnarray}}
\def\ba{\begin{array}}
\def\ea{\end{array}}

\def \L{\mathscr{L}}

\def \bdel{\overline{\partial}}

\def\Ric{\mathrm{Ric}}
\def\id{\mathrm{id}}
\def\be{\begin{equation}}
\def\ee{\end{equation}}
\def\tr{\mathrm{tr}}

\def\Sym{\mathrm{Sym}}
\def\SL2{\mathfrak{sl}_2(\bbC)} 

\def\SU{\mathrm{SU}}

\def\D{\mathbf{b}} 

\def\End{\mathrm{End}}

\def\vol{\mathrm{vol}}

\def\Sym{\mathrm{Sym}}

\def\T{T}

\def\grad{\mathrm{grad}}

\def \t5{\frac{1}{\sqrt{5}}}

\def\FN{\mathrm{FN}} 

\DeclareMathOperator{\TT}{TT} 

\DeclareMathOperator{\ric}{\mathrm{ric}}

\def \bfv{\mathbf{v}}

\DeclareMathOperator{\di}{d} 

\DeclareMathOperator{\bbC}{\mathbb{C}}

\newtheorem{pro}{Proposition}[section]
\newtheorem{teo}[pro]{Theorem}
\newtheorem{lema}[pro]{Lemma}
\newtheorem{coro}[pro]{Corollary}
\theoremstyle{definition}
\newtheorem{defi}[pro]{Definition}
\newtheorem{rema}[pro]{Remark}

\title{Einstein deformations of K\"ahler Einstein metrics}
\author{Paul-Andi Nagy}
\address[Paul-Andi Nagy]{Center for Complex Geometry \\
Institute for Basic Science(IBS)\\
55 Expo-ro, Yuseong-gu \\
34126 Daejeon, South Korea}
\date{\today}

\begin{document}

\begin{abstract}
We study Einstein deformations of negative K\"ahler Einstein metrics. 
We relate the second order Einstein deformation theory of negative K\"ahler-Einstein metrics to the complex geometry of the underlying 
K\"ahler manifold. After suitable gauge normalisation we show that the Taylor expansion to order two of an Einstein deformation tangent to $h_1$ in the infinitesimal deformation space is fully determined by $h_1^2$ and the divergence of the Kodaira-Spencer bracket $[h_1,h_1]^c$.
This substantially refines and extends recent results of Nagy-Semmelmann which state that Einstein deformations for negative K\"ahler-Einstein metrics are unobstructed to second order.

\medskip


\end{abstract}

\maketitle
\tableofcontents
\section{Introduction} \label{intro}
\subsection{Background and motivation}
Let $(M,g)$ be a compact Riemanian manifold where the metric $g$ is assumed to be Einstein i.e. $\ric^g=Eg$ for some $E$ in $\mathbb{R}$. We study the existence problem for Einstein deformations of $g$, that is curves $g_t$ of Riemannian metrics defined for short time $t$ such that 
$g_0=g$ and $\ric^{g_t}=Eg_t$. In addition we normalise such deformations to have  fixed volume i.e. $g_t \in 
\mathscr{M}_1$ where the latter space indicates the space of Riemannian metrics on $M$ having the same volume as $g$. A detailed exposition of the current state of the art of the Einstein deformations problem can be found  in \cite{S-Sw}; here we only recall a few facts relevant to this work.

The approach consists in considering the Taylor series expansion $g^{-1}g_t=\id+\sum \limits_{n \geq 1}\frac{t^n}{n!}h_n$ at $t=0$ and differentiating the Einstein equation $\ric^{g_t}=Eg_t$ with respect to $t$.  At first order this yields 
$$ h_1 \in \mathscr{E}(M,g)
$$
where the infinitesimal deformation space 
$ \mathscr{E}(M,g):=\{ h \in \ker \Delta_E : \delta^gh=0 \ \mathrm{and} \ \tr(h)=0\}$
and where $\Delta_E$ denotes the Einstein operator acting on symmetric $2$-tensors. 
In case the space $ \mathscr{E}(M,g)$ vanishes identically, the metric $g$ is {\it{rigid}}, in the sense of being isolated in the moduli space of Einstein metrics 
\cite{Besse}. 

Currently there is a good understanding and descriptions of $\mathscr{E}(M,g)$ for large classes of metrics; the list includes 
symmetric spaces (see \cite{S-Sw} for an overview), K\"ahler \cite{Koiso-I}, nearly-K\"ahler \cite{MoSe} and squashed $3$-sasaki structures \cite{NS-G2}. See also 
\cite{LL} for the case of homogeneous spaces and \cite{vanCo} for a discussion of the basic (with respect to the Reeb foliation) piece in $\mathscr{E}(M,g)$ for Sasaki-Einstein metrics. 

There are further obstructions on tensor $h_1$ as showed by Koiso \cite{Koiso}; indeed one must have $\int_{M}P(h_1)\vol=0$ where 
\begin{equation*}
2P(h_1) \;:= \; 3  g(\nabla^{2}_{e_i,he_i}h_1,h_1) \, - \, 6 g ((\nabla^2_{e_i,e_j}h_1)h_1e_i,h_1e_j) \, + \, 2 E \, \tr(h_1^3).
\end{equation*} 
This obstruction was used for the first time in \cite{Ba} in order to show that the bi-invariant Einstein metric on $\SU(2n+1)$ is rigid. A systematic approach to deriving the second order Einstein equation was undertaken in \cite{NS-E}; we have observed that the theory 
is governed by a second order differential operator build from the Fr\"olicher-Nijenhuis bracket $[\cdot, \cdot]^{\FN}$ on symmetric two 
tensors. Consider the operator $\bfv : \Gamma \left ( \Sym^2TM\right ) \oplus \Gamma \left ( \Sym^2TM\right ) \to 
\Gamma \left ( \Sym^2TM\right )$ determined from 
\begin{equation*}
\la \bfv(h_1,h_2),h_3 \ra_{L^2}=\mathfrak{S}_{abc} \la \delta^g[h_a,h_b]^{\FN},h_c \ra_{L^2}
\end{equation*}
where $\mathfrak{S}$ indicates the cyclic sum.
\begin{teo} \cite{NS-E}[Theorems 3.13 and 4.3]
The Einstein equation to second order is given by 
\begin{equation} \label{eqn2-o}
\widetilde{\Delta}_E(h_2-\tfrac{3}{2}h_1^2)=-\bfv(h_1,h_1)+Eh_1^2+\tfrac{1}{2}\delta^{\star_g}\tr(h_1^2).
\end{equation}
In particular the obstruction to second order deformation reads 
\begin{equation} \label{obs-o}
-\bfv(h_1,h_1)+Eh_1^2 \perp_{L^2} \mathscr{E}(M,g).
\end{equation}
\end{teo}
See section \ref{prel} for the definition of the perturbed Einstein operator $\widetilde{\Delta}_E$ and that of the operator $\delta^{\star_g}$. Note that the obstruction to deformation above contains Koiso's obstruction and also extends it in an intrinsic way. 
The obstruction in \eqref{obs-o} has been used \cite{NS-E, HaSe} to prove that the symmetric K\"ahler-Einstein metric on $\mathrm{Gr}_k(\bbC^{n})$ is rigid, when $n$ is odd. 

At the best of our knowledge there is currently no explicit description, amenable to direct computation, of the Einstein equation to order $3$.
\subsection{Main results} \label{intr-2}
Our first main result is that Einstein deformations $g_t \in \mathscr{M}_1$ can be normalised up to second order as follows, after taking 
into account the action of the gauge group 
$$ \mathbf{G}:=\{ f \in \mathrm{Diff}(M) : f^{\star}\vol=\vol\}.
$$

\begin{teo} \label{norm-thm}
Let $(M,g)$ be compact and Einstein and let $g_t \in \mathscr{M}_1$ be a family of Einstein metrics with Taylor expansion 
$$g^{-1}g_t=\id+th_1+\tfrac{t^2}{2}h_2+\cdots$$ at $t=0$. Up to a time dependent gauge 
transformation $g_t \mapsto f_t^{\star}g_t$ where $f_t \in \mathbf{G}$ we may assume that
\begin{itemize}
\item[(i)] the first order variation $h_1$ belongs to $\mathscr{E}(M,g)$
\item[(ii)] the second order variation $h_2$  satisfies 
\begin{equation} \label{o2-new}
\Delta_E(h_2-h_1^2)=\tfrac{1}{2}\widetilde{\Delta}_Eh_1^2+Eh_1^2-\bfv(h_1,h_1)
\end{equation}
as well as 
\begin{equation} \label{o2-no2new}
\delta^g(h_2-h_1^2)=-\frac{1}{4} \di\!\tr(h_1^2) \ \mathrm{and} \ \tr(h_2-h_1^2)=0.
\end{equation}
\end{itemize}
\end{teo}
\begin{rema} \label{rmk-2}
The construction the gauge transformation $f_t$ involves two main ingredients. Firstly, having the volume of $g_t$ constant 
forces $h_1$ and $h_2-h_1^2$ be trace free. Next, differentiating the gauge action $f_t \mapsto f_t^{\star}g_t$ to second order we observe that one can always find an $f_t$ such that the divergences of $h_1$ and $h_2-h_1^2$ are exact. In this case the Einstein equation in \eqref{eqn2-o} combined with trace properties of $\bfv$ entails that $\delta^gh_1=0$ and also \eqref{o2-no2new}.
\end{rema}
The rest of this paper focuses on Einstein deformations $g_t$ in $\mathscr{M}_1$ of K\"ahler Einstein metrics of negative scalar curvature; that is we assume $(M^{2m},g,J)$ is K\"ahler-Einstein with Einstein constant $E<0$. Following \cite{S-Sw} we briefly record that the Einstein 
deformation problem for $g$ is a priori unrelated to the complex deformation theory of $J$. 
\begin{rema} \label{rmk4bis}
If one knows that $J$ can be deformed to a family 
of complex structures $J_t$ then the metric $g$ can also be deformed to a family of Einstein metrics compatible with $J_t$, see 
\cite{Koiso-I}. On the other hand there are, see \cite{Dai}, many explicit examples of negative K\"ahler-Einstein manifolds where the deformation theory of the complex structure is obstructed. 
\end{rema}
In light of the above remark we make no assumption on the complex structure $J$ in the subsequent. Recently we have observed that 
\begin{teo}\cite{NS-E}[Theorem 5.3] \label{NS-E}
Assume that $(M^{2m},g,J)$ is compact and K\"ahler-Einstein with $E<0$. Then the obstruction in \eqref{obs-o} is always satisfied, in other words the Einstein deformation theory of $g$ is unobstructed to second order.
\end{teo}
The main result of this paper consists in considerably refining this result by 
showing how the normalised Einstein equation to second order \eqref{o2-new} can be solved explicitly.  We consider the splitting of the bundle of symmetric $2$-tensors
\begin{equation*} 
\Sym^{2}TM=\Sym^{2,+}TM \oplus \Sym^{2,-}TM
\end{equation*}
into $J$-invariant, respectively $J$-anti-invariant components. Whenever $h \in \Gamma \left (\Sym^2TM \right )$ we indicate with $h_{\pm}$ the corresponding components with respect to this splitting. Following \cite{Koiso-I} we also recall that due to $E<0$ the space $\mathscr{E}(M,g)$ 
coincides with the space of (normalised)infinitesimal complex deformations, that is 
\begin{equation} \label{red-E}
\mathscr{E}(M,g)=\{h \in \Gamma \left (\Sym^{2,-}TM \right ) : \bdel h=0 \ \mathrm{and} \ \delta^gh=0\}.
\end{equation}
\begin{teo} \label{main-1i}
Let $(M^{2m},g,J)$ be K\"ahler-Einstein with $E<0$ and let $g_t \in \mathscr{M}_1$ be an Einstein deformation of $g$ with Taylor expansion $g^{-1}g_t=\id+th_1+\tfrac{t^2}{2}h_2+o(t^3)$ at $t=0$. Up to a gauge transformation $g_t \mapsto f_t^{\star}g_t$ with $f_t \in \mathbf{G}$ as given in Theorem \ref{norm-thm} we have
\begin{itemize}
\item[(i)] $h_2^{+}=h_1^2$
\item[(ii)] $ h_2^{-}=\mathbf{h}_2-\tfrac{1}{2}\L_{J\grad \mathbf{f}}J $ where the pair $(\mathbf{h}_2, \mathbf{f})$ belongs to $\TT^{-}(M,g) \oplus C^{\infty}M$ and satisfies the equations 
\begin{equation*}
\begin{split}
&\Delta_E^g \mathbf{h}_2=-2\delta^g[h_1,h_1]^c\ \mathrm{and} \ (\Delta^g-2E)\mathbf{f}=-\tfrac{1}{2}\tr(h_1^2).
\end{split}
\end{equation*}
\end{itemize}
In particular $h_2-h_1^2$ belongs to $\Gamma \left (\Sym^{2,-}TM \right )$ and satisfies $\bdel (h_2-h_1^2)=\bdel \mathbf{h}_2$.
\end{teo}
Above we have indicated with $\TT^{-}(M,g)=\TT(M,g) \cap \Gamma \left (\Sym^{2,-}TM \right )$
the space of $J$-anti-invariant $\TT$ tensors and with 
$[\cdot ,\cdot]^c$ the Kodaira-Spencer bracket, see section \ref{cpx-bra} for definitions and properties. A few remarks 
are now in order as far the construction of the pair 
$(\mathbf{h}_2, \mathbf{f})$ is concerned.
\begin{rema} \label{rmk-3}
Since $h_1$ is divergence free  $\delta^g[h_1,h_1]^c$ belongs to $\TT^{-}(M,g)$ (see Proposition \ref{sym-N} in the paper); in addition 
the description of $\mathscr{E}(M,g)$ in \eqref{red-E} leads easily to having 
$\delta^g[h_1,h_1]^c \perp_{L^2} \mathscr{E}(M,g)=\ker \Delta_E$. These facts grant existence for $\mathbf{h}_2 \in \TT^{-}(M,g)$ solving the equation in (ii) above. Since $E<0$ the operator $\Delta-2E$ is invertible on functions; this guarantees existence for the function $\mathbf{f}$. 
\end{rema}

We found it striking that the solution $h_2^{+}$ assumes a particularly simple
algebraic form, and also that the final form of the Einstein equation reduces to the divergence of the Kodaira-Spencer bracket alone, instead of the symmetrised operator $\bfv$. 
We believe that Theorem \ref{main-1i} paves the way to understanding the third order Einstein deformation theory for negative K\"ahler-Einstein metrics.
\subsection{Key arguments and outline of proofs for K\"ahler metrics} \label{key1}
In section \ref{cpx-bra} we use representation theoretical arguments in order to obtain a comparison formula relating the 
Fr\"olicher-Nijenhuis and Kodaira-Spencer brackets on K\"ahler manifolds. This formula is described in full generality in Proposition \ref{com-brak}; it also leads to the important observation that the divergence of the Kodaira-Spencer bracket acting on divergence free 
tensors is a $\TT$-tensor (see Proposition \ref{sym-N}). Yet another direct consequence of Proposition \ref{com-brak} is to show that 
the restriction of $\bfv$ to $\Gamma \left (\Sym^{2,-}TM \right )$ is determined, up to divergence terms, by the Kodaira-Spencer bracket; this is Proposition \ref{3-} in the paper. The latter actually allows computing the component on $\Gamma \left (\Sym^{2,-}TM \right )$ of elements of the form $\bfv(h,h)$ with $h$ in $\mathscr{E}(M,g)$, as given by \eqref{red-E}. Furthermore we observe that 
the computation of the component on $\Gamma \left (\Sym^{2,+}TM \right )$ of $\bfv(h,h)$ reduces to that of $\la H \sharp \di_{\nabla^g}h, \di_{\nabla^g}h \ra_{L^2}$ with $H \in\Gamma \left (\Sym^{2,+}TM \right )$; this is the main technical challenge in the second part of this paper which is completed 
in Proposition \ref{type-I} after a series of preliminary lemmas and representation theoretical arguments. The final expression for $\bfv(h,h)$ is then given in Proposition \ref{typeI-v} which allows bringing equation \eqref{o2-new} to final form. Theorem \ref{main-1i} is then proved by Hodge theoretical arguments for the Einstein operator $\Delta_E$ and by also using the normalisation to second order developed in Theorem \ref{norm-thm}.
\subsection*{Acknowledgments}
Paul-Andi Nagy was supported by the Institute for Basic Science (IBS-R032-D1). 
\section{Preliminaries} \label{prel}
Let $(M^n, g)$ be a Riemannian manifold and denote with $\nabla^g$ the Levi-Civita connection of $g$ acting on $TM$ and all
tensor bundles. For the Riemannian curvature tensor $R^g$ we use the convention $R^g(X,Y)=(\nabla^g)^2_{Y,X}-(\nabla^g)^2_{X,Y}$
for tangent vectors $X, Y \in TM$. 
In the following we will make  a notational difference between the Ricci form $\ric^g$ considered as a symmetric bilinear form and the Ricci tensor $\Ric^g$  
considered as a symmetric endomorphism. We will denote by $g^{-1}$ the isomorphism, induced by the metric, between symmetric bilinear forms and symmetric
endomorphisms, e.g. we have $g^{-1} \ric^g = \Ric^g$. The subbundle of $g$-symmetric endomorphisms is denoted by $\Sym^2 TM \subseteq \End(TM)$. It is preserved 
by $\nabla^g$ and  the curvature action $h \mapsto \ring{R}h:=\sum_i R^g(e_i,\cdot)he_i$, where $\{ e_i\}$ is some $g$-orthonormal frame. 

We will use the coupled exterior differential $\di_{\nabla^g}:\Omega^{k}(M,\T M) \to \Omega^{k+1}(M,\T M)$. For $k=0$, i.e. on vector fields,
it coincides with the covariant derivative $\nabla^g$. For $k=1$, that is on endo\-morphisms $h$ considered as elements in $ \Omega^1(M,TM)$, it  
can be written as
$\di_{\nabla^g}h(X,Y)=(\nabla^g_Xh)Y-(\nabla_Y^gh)X$. The operator formally  adjoint to $\di_{\nabla^g}$  is the divergence operator  $\delta^g:\Omega^{k}(M,TM) \to \Omega^{k-1}(M,TM)$ which is defined according to the convention $\delta^g = -  \sum_i e_i \lrcorner \nabla^g_{e_i}$. 
%

%
%

Restricting the divergence to symmetric $2$-tensors we have $\delta^g : \Gamma(\Sym^2TM) \to  \Omega^1M$. Its formal adjoint 
$\delta^{\star_g} : \Omega^1M \to \Gamma(\Sym^2TM)$ is the symmetric part of $\nabla^g$, i.e.  $\nabla^g=\delta^{\star_g}+\frac{1}{2}\di$ on $\Omega^1M$.
As a consequence we note for later application that $\tr (\delta^{\star_g} \alpha)  = - \di^{\star}\!\alpha$ holds for any $1$-form $\alpha$.
We also have the well-known formula  $\delta^{\star_g}\alpha =\frac{1}{2}g^{-1}\L_{\alpha^{\sharp}}g$, characterising the kernel of $\delta^{\star_g}$
as $1$-forms dual to Killing vector fields.

In dealing with the trace and divergence of symmetric tensors we will use the Bianchi operator $\D^g: \Gamma(\Sym^2TM) \to \Omega^1M$ given by  $\D^g=2\delta^g+\di\! \circ \tr$. 
Certainly $\D^g$ vanishes on the space of the so-called $\TT$-tensors defined by
$$
\TT(g):=\{h \in \Gamma(\Sym^2TM) : \delta^g h=0  \ \mathrm{and} \ \tr(h)=0\} . 
$$
However terms involving divergence and trace must be kept track of in order to obtain gauge invariant equations. 
For latter use recall that $\D^g \, \Ric^g =0$, which is a consequence of the differential Bianchi identity 
(see \cite{Besse}, 12.33).

On symmetric $2$-tensors and for Einstein metrics $g$ with Einstein constant $E$ we will consider the Einstein operator 
$\Delta_E := (\nabla^g)^{\star}\nabla^g-2\ring{R} $. Note that $\Delta_E = \Delta_L - 2E$, where $\Delta_L$ is the Lichnerowicz
Laplacian on symmetric $2$-tensors (see \cite{Besse}, 1.143). We will also use a perturbation of the Einstein operator, 
the differential operator $\widetilde{\Delta}_E:
\Gamma(\Sym^2TM) \to \Gamma(\Sym^2TM)$ given by 
\begin{equation}\label{D-tilde}
 \widetilde{\Delta}_E \;:= \;  \Delta_E \, - \, 2\delta^{\star_g} \circ (\delta^g+\frac{1}{2}\di \circ \tr)\; = \; \Delta_E \, - \, \delta^{\star_g} \circ \D^g.
\end{equation}

The following Weitzenb\"ock formula on  $2$-tensors  
(see \cite{NS-E}) will be used at several  places in this paper.
\begin{equation} \label{wz1}
\delta^g \circ \di_{\nabla^g} \;=\;  -\nabla^g \circ \delta^g  \, + \, \Delta_E \, + E+\,  \ring{R} \; =\;  - (\delta^{\star_g}+\frac{1}{2}\di) \circ \delta^g\, +\, \Delta_E\, +E+\, \ring{R} .
\end{equation}
Further properties of the Einstein operator needed in this paper are summarised below.
\begin{lema} \label{E-div}
The following hold
\begin{itemize}
\item[(i)] we have $\Delta_E \circ \delta^{\star_g}=\delta^{\star_g} \circ (\Delta^g-2E)$
\item[(ii)] $\delta^g \circ \Delta_E=(\Delta^g-2E) \circ \delta^g$.
\end{itemize}
\end{lema}
\begin{proof}
(i) We have $\widetilde{\Delta}_E \circ \delta^{\star_g}=0$ by \cite[eqn. (28)]{NS-E} as well as $\mathbf{b}^g \circ \delta^{\star_g}=\Delta^g-2E$ by \cite[eqn. (26)]{NS-E}. The claim follows from $\widetilde{\Delta}_E=\Delta_E-\delta^{\star_g} \circ \mathbf{b}^g$. \\
(ii) follows from (i) by $L^2$-duality.
\end{proof}
We still need another Weitzenb\"ock form, this time on $1$-forms, which reads 
\begin{equation} \label{wz2}
2 \delta^g \delta^{\star_g}  \, - \,   \di \di^{\star}  \; = \;  \Delta^g \, - \,  2E ,
\end{equation}
where $\Delta^g = \di \di^{\star} + \di^{\star} \di$ is the Hodge Laplacian and $g$ is again an Einstein metric; see \cite{NS-E} for further details. In subsequent trace computations we will also need the following 
\begin{lema} \label{tr-00}
Let $H$ belong to $\Gamma \left (\Sym^2TM \right )$. Then 
\begin{itemize}
\item[(i)] $\tr \delta^g\di_{\nabla^g}H=\di^{\star}(\delta^gH+\di\!\tr(H))$
\item[(ii)] $\tfrac{1}{2}\tr \widetilde{\Delta}_E H=\di^{\star}\left (\delta^gH+\di\!\tr(H) \right )-E\tr(H)$.
\end{itemize}
\end{lema}
\begin{proof}
Record the identity $\tr \circ \delta^{\star_g}=-\di^{\star} $ and also that $\tr \ring{R}H=E \tr(H)$. Both claims follow now from \eqref{wz1}.
\end{proof}

\subsection{The Fr\"olicher-Nijenhuis bracket} \label{FN-br}

On a given manifold $M$ we first recall that the Fr\"olicher-Nijenhuis bracket  for sections $h$ of $\End(TM)=\Lambda^1(M,TM)$ reads 
$$
[h,h]^{\FN}(X,Y) \, = \, -h^2[X,Y] \, + \, h([hX,Y] \, + \, [X,hY]) \, - \, [hX,hY]  .
$$  
See also \cite{Kollar}[section 8] for a detailed discussion of the Fr\"olicher-Nijenhuis bracket of arbitrary degree forms in $\Omega^{\star}(M,TM)$. From now on assume that $g$ is a Riemannian metric on $M$. Recast in terms of the Levi-Civita connection $\nabla^g$ of $g$ the bracket $[h,h]^{\FN}$ reads 
\begin{equation} \label{bra-na}
[h,h]^{\FN}(X,Y) \, = \, -(\nabla^g_{hX}h)Y \, + \, (\nabla^g_{hY}h)X  \, + \, (h \circ \di_{\nabla^g}h)(X,Y) .
\end{equation}
Here  $\di_{\nabla^g}$ acts on $h$  considered as  an element of $\Omega^1(M, TM)$. In the following, whenever the pair $(h, \alpha)$ belongs 
to $\Sym^2TM \oplus \Lambda^2(M,TM)$ we indicate with $h \circ \alpha \in \Lambda^2(M,TM)$ the tensor given by 
$(X,Y) \mapsto h \alpha(X,Y)$.
A more concise version of \eqref{bra-na} is  
\begin{equation} \label{bra-nac}
[h,h]^{\FN} \, = \, -h\sharp \di_{\nabla^g}h \, + \, \di_{\nabla^g}h^2
\end{equation}
where the algebraic action $\alpha \in \Lambda^2(M,TM) \mapsto h \sharp \alpha \in \Lambda^2(M,TM)$ is defined according to $h \sharp \alpha(X,Y)=\alpha(hX,Y)+\alpha(X,hY)$.
The Fr\"olicher-Nijenhuis bracket is extended to a symmetric bracket on the space  $\Gamma(\Sym^2TM )$ via 
\begin{equation} \label{bra-nac2}
2[h_1,h_2]^{\FN} \, = \, -(h_1 \sharp \di_{\nabla^g}h_2 \, + \, h_2 \sharp \di_{\nabla^g}h_1) \, + \, \di_{\nabla^g} \{h_1,h_2\}
\end{equation}
where the anti-commutator $\{h_1,h_2\}:=h_1 \circ h_2+h_2 \circ h_1$. In particular we have 
$[\id ,h]^{\FN}=0$ for all $h \in \End(TM)$. Finally we record that the Fr\"olicher-Nijenhuis bracket is invariant under the action of the diffeomorphism group, $\varphi_{\star}[h_1,h_2]^{\FN}=[\varphi_{\star}h_1, \varphi_{\star}h_2]^{\FN}$ whenever $\varphi \in \mathrm{Diff}(M)$; with respect to the Lie derivative we thus must have 
$$ \L_X[h_1,h_2]^{\FN}=[\L_Xh_1,h_2]^{\FN}+[h_1,\L_Xh_2]^{\FN}
$$
whenever $X \in \Gamma(TM)$. 
\section{Bianchi map normalisation to second order} \label{Bianchi}
Let $(M,g)$ be a compact Riemannian manifold and $g_t$ be a time dependent family of Riemannian metrics. It is well known
that, up to a suitable gauge transformation, the tensor $h_1:=g^{-1}\dot{g}_{\vert t=0}$ can be normalised to belong to $\TT(M,g)$. 
This process involves first reduction to metrics of fixed volume $\vol(g_t)=\vol(g)$; this can always be achieved by Moser's theorem which asserts that any two volume forms are equivalent. The next step is to take into account the action of the gauge group 
$$ \mathbf{G}:=\{f \in \mathrm{Diff}(M) : f^{\star}\vol=\vol\}.
$$
When coupled with the Einstein equation to first order, that is $\widetilde{\Delta}_Eh_1=0$ this ensures reduction to $\delta^g h_1=0$ and further $h_1 \in \mathrm{E}(g)$.

In this section we considerably extend extend these ideas to normalise the action of trace and divergence operators on 
the second order variation of $g_t$. During this process we obtain an easier to deal with form of the second order Einstein 
equation found in \cite{NS-E}. The technical difficulties that arise are related to the computation of 
$$ \tr \ \bfv(h,h) $$
where $h \in \mathscr{E}(M,g)$. This computation is performed in detail in the section as follows.

Throughout this paper we deal solely with Einstein deformations $g_t$ with constant volume, that is $\vol(g_t)=\vol(g)$ for small 
$t$.
\subsection{Role of the trace} \label{tr-sec}
We start by examining further properties of the Fr\"olicher-Nijenhuis bracket, in particular up to which extent 
$\delta^g [h,h]^{\FN}$, where $h \in \Sym^2TM$, fails to be symmetric. We define the symmetric operator according to 
$$ \mathrm{L}(h):=\nabla_{e_i}^g h \circ \nabla^g_{e_i}h
$$
and observe that 
\begin{lema} \label{sym-bra}
Letting $h$ be in $ \Gamma \left (\Sym^2TM \right )$ the following hold 
\begin{itemize}
\item[(i)] the tensor $\delta^g\left ([h,h]^{\FN}-h \circ \di_{\nabla^g}h \right )- (\ring{R}h) \circ h-\nabla^g(\delta^gh) \circ h$ is symmetric
\item[(ii)] we have 
\begin{equation*}
\begin{split}
\la \delta^g [h,h]^{\FN},H \ra_{L^2}=&\la h \circ \di_{\nabla^g}h, \di_{\nabla^g}H \ra_{L^2}+
\la \di_{\nabla^g}h, H \circ \di_{\nabla^g}h \ra_{L^2}\\
&-\la \nabla^g_{he_i}h, \nabla^g_{e_i}H \ra_{L^2}\\
&+\la (\ring{R}h) \circ h-Eh^2-\mathrm{L}(h),H\ra_{L^2}+\la -\nabla^g_{\delta^gh}h+\nabla^g(\delta^gh) \circ h, H\ra_{L^2}
\end{split}
\end{equation*}
for all $H \in \Gamma \left (\Sym^2TM \right )$.
\end{itemize}
\end{lema}
\begin{proof}
Differentiating in  \eqref{bra-na} shows that 
\begin{equation*}
\begin{split}
\delta^g([h,h]^{\FN}-h \circ \di_{\nabla^g}\! h)_X=((\nabla^g)^2_{e_i,he_i} h)X-(\nabla^g_{\delta^gh}h)X-((\nabla^g)^2_{e_i,hX}h)e_i-(\nabla^g_{(\nabla^g_{e_i}h)X}h)e_i
\end{split}
\end{equation*}
for all $X$ in $TM$. Using the Ricci identity leads to 
$$((\nabla^g)^2_{e_i,hX}h)e_i=-\nabla^g_{hX}\delta^gh-[R^g(e_i,hX),h]e_i=-\nabla^g_{hX}\delta^gh+(Eh^2-(\ring{R}h) \circ h )X.$$
Taking $Y$ in $TM$ we furthermore compute 
\begin{equation*}
\begin{split}
g((\nabla^g_{(\nabla^g_{e_i}h)X}h)e_i,Y)=&g((\nabla^g_{e_i}h)X,e_k)g((\nabla^g_{e_k}h)e_i,Y)\\
=&g(\di\!_{\nabla^g}h(e_i,e_k),X)g((\nabla^g_{e_k}h)e_i,Y)+g(\mathrm{L}(h)X,Y)\\
=&-\tfrac{1}{2}g(\di_{\nabla^g}\!h(e_i,e_k),X)g(\di_{\nabla^g}\!h(e_i,e_k),Y)+g(\mathrm{L}(h)X,Y).
\end{split}
\end{equation*}
Gathering terms thus leads to 
\begin{equation*}
\begin{split}
g \left (\delta^g\left ([h,h]^{\FN}-h \circ \di\!_{\nabla^g} h \right)_X, Y \right )=&g \left (((\nabla^g)^2_{e_i,he_i}h)X-(\nabla^g_{\delta^gh}h)X+\nabla^g_{hX}\delta^gh,Y \right )\\
+&g \left (\left ((\ring{R}h) \circ h-Eh^2-\mathrm{L}(h) \right )X,Y \right )\\
+&\tfrac{1}{2}g(\di_{\nabla^g}\!h(e_i,e_k),X)g(\di_{\nabla^g}\!h(e_i,e_k),Y)
\end{split}
\end{equation*}
Since the last summand above is symmetric in $X$ and $Y$ the claim in (i) follows from the symmetry of the operator $\mathrm{L}(h)$. The claim in (ii) follows 
by taking $X=e_k, Y=He_k$ and integrating by parts; that is 
$ \la  (\nabla^g)^2_{e_i,he_i}h, H \ra_{L^2}=-\la \nabla_{he_i}^gh,\nabla_{e_i}^gH \ra_{L^2}$.

\end{proof}
In the second part of the paper we will also use Lemma \ref{sym-bra} for K\"ahler manifolds in order to determine the symmetry properties of the complex Fr\"olicher-Nijenhuis bracket. To that extent we observe that Lemma \ref{sym-bra} can be brought 
to the following more concise form. Whenever $(X,h)$ belongs to $TM \oplus \Sym^2TM$ the wedge product $X \wedge h \in \Lambda^{2}(M,TM)$ is defined according to 
$(X \wedge h)(Y,Z)=g(X,Y)hZ-g(X,Z)hY$. 
\begin{pro} \label{sym-branN}
Let $h$ belong to $\Sym^2TM$. The tensor 
$$\delta^g\left ([h,h]^{\FN}+\delta^gh \wedge h-h \circ \di_{\nabla^g}h \right )- (\ring{R}h) \circ h \ \mathrm{is \ symmetric}.$$
\end{pro}
\begin{proof}
Whenever the tensor $Q$ belongs to $ \End{TM}$ we indicate with $Q_{\mathrm{skew}}$ its skew-symmetric component which is determined from 
$g(Q_{\mathrm{skew}}X,Y)=\tfrac{1}{2}(g(QX,Y)-g(X,QY))$.
A straightforward computation leads to 
$$\delta^g(\delta^gh \wedge h)=(\di^{\star}\delta^gh)h-\nabla^g_{\delta^gh}h+h \circ (\nabla^g(\delta^gh))^T-\delta^gh \otimes \delta^gh$$
where the superscript indicates the transpose with respect to the metric $g$. In particular 
$$(\delta^g(\delta^gh \wedge h))_{\mathrm{skew}}=(h \circ (\delta^{\star_g}\delta^gh-\tfrac{1}{2}\di\!\delta^gh))_{\mathrm{skew}}=
\tfrac{1}{2} \left ([h,\delta^{\star_g}\delta^gh]-\tfrac{1}{2}\{h,\di\!\delta^gh)\} \right ).
$$
At the same time $(\nabla^g(\delta^gh) \circ h)_{\mathrm{skew}}=((\delta^{\star_g}\delta^gh+\tfrac{1}{2}\di\!\delta^gh) \circ h)_{\mathrm{skew}}=\tfrac{1}{2} \left ([\delta^{\star_g}\delta^gh,h]+\tfrac{1}{2}\{h,\di\!\delta^gh)\} \right )$. Comparison of the skew-symmetric components entails that the tensor $\delta^g(\delta^gh \wedge h)+\nabla^g(\delta^gh) \circ h$ is symmetric and the claim follows from  Lemma \ref{sym-bra}, (i).
\end{proof}
Next we observe that one can also compute explicitly 
the trace of $\delta^g[h,h]^{\FN}$, a fact that will be needed later on.
\begin{lema} \label{tr-12}
Assume that $h$ in $\Gamma \left (\Sym^2TM \right )$ satisfies $\tr(h)=0$.
Then 
$$\tr \ \delta^g[h,h]^{\FN}=\tfrac{1}{2}\Delta^g \tr(h^2).$$
\end{lema}
\begin{proof}
Using \eqref{bra-na} shows that 
\begin{equation*}
\begin{split}
g([h,h]^{\FN}(X,e_i),e_i)=&-g((\nabla^g_{hX}h)e_i,e_i)+g((\nabla^g_{he_i}h)X,e_i)+g((\nabla^g_{X}h)e_i-(\nabla^g_{e_i}h)X,he_i)\\
=&-g((\nabla^g_{hX}h)e_i,e_i)+g((\nabla^g_{X}h)e_i,he_i)
\end{split}
\end{equation*}
since $h$ is symmetric. Because $h$ is trace free it follows that 
$g([h,h]^{\FN}(X,e_i),e_i)=\tfrac{1}{2}\L_X \tr(h^2)$. The claim follows by taking $X=e_k$ and differentiating in direction $e_k$.
\end{proof}
We finish this section by computing the trace of the operator $\bfv(h,h)$ with $h \in \mathscr{E}(M,g)$. To perform this 
computation we first recall that 
\begin{equation} \label{bfv-1}
\begin{split}
\la \bfv(h_1,h_2),H\ra_{L^2}=&2\la [h_1,h_2]^{\FN}, \di_{\nabla^g}H \ra_{L^2}-\la H \sharp \di\!_{\nabla^g}h_1,\di \!_{\nabla^g}h_2\ra_{L^2}\\
+&\tfrac{1}{2} \la \{h_1,\delta^g \di\!_{\nabla^g}h_2\}+\{h_2,\delta^g \di\!_{\nabla^g}h_1\}-\delta^g\di_{\nabla^g}\{h_1,h_2\},H \ra_{L^2}
\end{split}
\end{equation}
for symmetric tensors $h_1,h_2$ and $H$ in $\Gamma \left (\Sym^2TM \right )$; see the proof of Theorem 4.4 in \cite{NS-E}.  Record now the following preliminary 
\begin{lema} \label{tr-23}
Assume that the pair $(h,f) \in \mathscr{E}(g) \oplus C^{\infty}M$. Then
$$\la \di_{\nabla^g}\!h,f\di_{\nabla^g}\!h \ra_{L^2}=\la fh,Eh+\ring{R}h\ra_{L^2}-\tfrac{1}{2} \la \di\!f, \mathbf{b}(h^2) \ra_{L^2}.$$

\end{lema}
\begin{proof}
Start from $\di_{\nabla^g}(fh)=f\di\!_{\nabla^g}h+\di\!f \wedge h$ so that 
$$ \la \di\!_{\nabla^g}h,f\di\!_{\nabla^g}h \ra_{L^2}=\la fh, \delta^g \di\!_{\nabla^g}h \ra_{L^2}-\la \delta^g(\di\!f \wedge h), h \ra_{L^2}.
$$
Also observe that $\delta^g(\di\!f \wedge h)=(\Delta f)h-\nabla^g_{\grad f}h+
h \circ \nabla^g\grad f$ since $\delta^gh=0$. It follows that 
$$ \la \delta^g(\di\!f \wedge h),h\ra_{L^2}=\la \Delta^g f, \tr(h^2) \ra_{L^2}-\tfrac{1}{2} \la \di\!f, \di\tr(h^2) \ra_{L^2}+
\la \di f, \delta^gh^2\ra_{L^2}=\tfrac{1}{2} \la \di\!f, \mathbf{b}(h^2) \ra_{L^2}.
$$
The claim follows by gathering terms and also using \eqref{wz1}.
\end{proof}
Finally, we are able to prove the following 
\begin{pro} \label{tr-v}
Assume that $h$ belongs to $\mathscr{E}(M,g)$. We have 
$$\tr \left ( Eh^2-\bfv(h,h)+\tfrac{1}{2}\widetilde{\Delta}_Eh^2 \right )=0.$$
\end{pro}
\begin{proof}
We use \eqref{bfv-1} whilst taking into account the comparison formula \eqref{wz1}; we find 
\begin{equation*}
\begin{split}
\la \bfv(h,h), f\id \ra_{L^2}=&2\la \delta^g[h,h]^{\FN}, f\id\ra_{L^2}-2\la f\di_{\nabla^g}h,\di_{\nabla^g}h\ra_{L^2}\\
&+\la 
\{h,Eh+\ring{R}h\}-\delta^g\di\!_{\nabla^g}h^2, f\id \ra_{L^2}.
\end{split}
\end{equation*} 
After taking into account Corollary \ref{tr-12} as well as Lemma \ref{tr-23} an $L^2$-orthogonality argument shows that 
$ \tr \ \bfv(h,h)=\Delta^g \tr(h^2)+\di^{\star} \mathbf{b}(h^2)-\tr(\delta^g \di_{\nabla^g}h^2).
$ 
The claim follows now from the first identity in Lemma \ref{tr-00}.
\end{proof}
\subsection{Proof of normalisation to second order} \label{nos}
We start from the following elementary observation.
\begin{lema} \label{red-1}
Assume that $H \in \Gamma \left (\Sym^2TM \right )$. There exists $X \in \Gamma (TM)$ such that $\di^{\star_g}X=0$ and 
$\delta^g(H+\delta^{\star_g}X) \in \mathrm{Im} \di$.
\end{lema}
\begin{proof}
Split $H=Q+\delta^{\star_g}X+f\id $ where $Q \in \TT(M,g)$ and $(X,f) \in \Gamma(TM) \oplus C^{\infty}M$; see e.g. \cite{Besse} for more details. Furthermore split, by Hodge decomposition,  
$X=X_0+\di\! q$ where $X_0$ is divergence free and $q$ is a function on $M$. Using \eqref{wz2} shows that 
$$ \delta^g(H-\delta^{\star_g}X_0)=\di((\Delta^g-E)q-f)
$$
is exact and the claim is proved by letting $X:=-X_0$.
\end{proof}
Yet another preliminary  observation we need is the following 
\begin{lema} \label{red-2}
Assume that the family of metrics $g_t \in \mathscr{M}_1$ with $g_0=g$ has Taylor expansion 
$$g^{-1}g_t=\id+th_1+\tfrac{t^2}{2}h_2+\cdots$$ 
at $t=0$. Then the tensors  
$ h_1$ and $h_2-h_1^2$ are both trace-free.
\end{lema} 
\begin{proof}
Write $g^{-1}g_t:=h_t$ so that having constant volume amounts to $\det h_t=1$. Differentiating with respect to $t$ yields 
$\tr(h_t^{\prime}h_t^{-1})=0$ hence taking $t=0$ shows that $\tr(h_1)=0$. Differentiating a second time whilst taking into account that 
$(h_t^{-1})^{\prime}=-h_t^{-1}h_t^{\prime}h_t^{-1}$ yields $\tr(h_2-h_1^2)=0$. 
\end{proof} 
In particular having $ h_1$ and $h_2-h_1^2$ trace-free is a property which is invariant under the gauge group $\mathbf{G}$, since the latter preserves the volume form.

The main result in this section is that 
the coefficients in the Taylor expansion of the Einstein metric $g_t$ can be normalised--up to gauge action--to second order as well. 
The idea is to differentiate to second order the action of the gauge group $\mathbf{G}$ on a given one parameter family of metrics 
$g_t$, then use Lemma \ref{red-1} to reduce to the case when the divergences of $h_1$ and $h_2-h_1^2$ are exact.
\begin{pro} \label{norm-2nd}
Let $g_t \in \mathscr{M}_1$ be a family of Riemannian metrics with Taylor expansion 
$$g^{-1}g_t=\id+th_1+\tfrac{t^2}{2}h_2+\cdots$$ at $t=0$. Up to a gauge 
transformation i.e. $g_t \mapsto f_t^{\star}g_t$ where $f_t \in \mathbf{G}$  we may assume that the divergences 
$$ \delta^g h_1 \ \mathrm{and} \ \delta^g(h_2-h_1^2) \ \mathrm{are \ both  \ exact}. 
$$
\end{pro}
\begin{proof}
We look for a time dependent family $f_t \in \mathbf{G}$ with $f_0=\id$ such that the coefficients of the Taylor expansion of the family of metrics $ \tilde{g}_t:=f_t^{\star}g_t$
have the required properties. Consider 
the time-dependent family of vector fields $X_t$  determined from $f_t^{\prime}=X_t \circ f_t$; since $f_t \in \mathbf{G}$ we must have 
$\di^{\star_g}X_t=0$. Indeed, differentiating in $f_t^{\star}\vol=\vol$ we get 
$\L_{X_t}\vol=0$, that is $X_t$ is divergence free. Moreover, we consider the Taylor expansion to order $2$ of $\tilde{g}_t$, that is 
$g^{-1}\tilde{g}_t=\id+t\tilde{h}_1+\tfrac{t^2}{2}\tilde{h}_2+o(t^3).$ Differentiating shows that 
$$\tilde{g}_t^{\prime}=f_t^{\star} \left (g_t^{\prime}+\L_{X_t}g_t \right ).
$$
In particular  we have 
$$\tilde{h}_1=h_1+2\delta^{\star_g}X_0.$$
Differentiating a second time shows that, at $t=0$, 
\begin{equation*}
\begin{split}
\tilde{g}_t^{\prime \prime}=g_t^{\prime \prime}+2\L_{X_0}g_t^{\prime}+\L_{X_0}^2g+\L_{X_1}g
\end{split}
\end{equation*}
where $X_1:=X_t^{\prime}(0)$. Thus the Taylor coefficient 
$$ \tilde{h}_2=h_2+q+2\delta^{\star_g}X_1
$$
where the symmetric tensor $q:=g^{-1}(2\L_{X_0}g_t^{\prime}+\L_{X_0}^2g)$ only depends on $X_0$ and $g^{\prime}(0)$; note there is no need of making 
this tensor more explicit. Therefore 
$$\tilde{h}_2-\tilde{h}_1^2=h_2+q-\tilde{h}_1^2+2\delta^{\star_g}X_1.
$$
Using Lemma \ref{red-1} for the symmetric tensor $h_1$ yields a divergence free vector field $X_0$ such that 
$\tilde{h}_1$ has exact divergence. Similarly, using Lemma \ref{red-1} for the symmetric tensor $h_2+q-\tilde{h}_1^2$ thus produces a divergence free vector field 
$X_1$ such that $h_2+q-\tilde{h}_1^2+2\delta^{\star_g}X_1$, and hence $\tilde{h}_2-\tilde{h}_1^2$, has exact divergence. 

To conclude, we consider the divergence free family $X_t:=X_0+tX_1$ where the vector fields $X_0$ and $X_1$ are constructed as above. The desired gauge transformation $f_t \in \mathbf{G}$ 
is obtained by solving $\dot{f}_t=X_t \circ f_t$.
\end{proof}
When coupled with the Einstein equations to first and second order the normalisation in the previous proposition yields the 
following.
\begin{teo} \label{norm-thm}
Let $g_t$ be a family of Einstein metrics with $\vol(g_t)=\vol$, with Taylor expansion 
$$g^{-1}g_t=\id+th_1+\tfrac{t^2}{2}h_2+\cdots$$ at $t=0$. Up to a time dependent gauge 
transformation $g_t \mapsto f_t^{\star}g_t$ where $f_t \in \mathbf{G}$ we can assume that
\begin{itemize}
\item[(i)] the first order variation $h_1$ belongs to $\mathscr{E}(M,g)$
\item[(ii)] the second order variation $h_2$  satisfies 
\begin{equation} \label{o2-no}
\Delta_E(h_2-h_1^2)=\tfrac{1}{2}\widetilde{\Delta}_Eh_1^2+Eh_1^2-\bfv(h_1,h_1).
\end{equation}
as well as 
\begin{equation} \label{o2-no2}
\delta^g(h_2-h_1^2)=-\frac{1}{4} \di\!\tr(h_1^2) \ \mathrm{and} \ \tr(h_2-h_1^2)=0.
\end{equation}
\end{itemize}
\end{teo} 
\begin{proof}
The argument uses the normalisation from Proposition \ref{norm-2nd}; therefore we assume that the trace-free symmetric tensors 
$h_1,h_2-h_1^2$ have exact divergences.\\
(i) This is well known but we include the proof since the argument fits the higher order set-up we are dealing with. Since $g_t$ is an Einstein metric, differentiating the Einstein equation to first order yields 
$\widetilde{\Delta}_Eh_1=0$, see e.g. \cite{NS-E}[Proposition 3.2]. We apply the trace and take into account Lemma \ref{tr-00},(ii); since $h_1$ is trace free 
it follows that $\di^{\star}\delta^gh_1=0$. Since $\delta^gh_1$ is exact this leads to having $\delta^gh_1=0$, as claimed.\\
(ii)Since $h_1 \in \TT(M,g)$ by (i) we have 
$$\widetilde{\Delta}_E(h_2-\tfrac{3}{2}h_1^2)=Eh_1^2-\bfv(h_1,h_1)+\tfrac{1}{2}\delta^{\star_g} \di\!\tr(h_1^2)
$$ 
by \cite{NS-E}[Theorem 3.13]. Equivalently 
$$\widetilde{\Delta}_E(h_2-h_1^2)=\tfrac{1}{2}\widetilde{\Delta}_Eh_1^2+Eh_1^2-\bfv(h_1,h_1)+\tfrac{1}{2}\delta^{\star_g} \di\!\tr(h_1^2).
$$ 
The assumption of having constant volume guarantees that $h_2-h_1^2$ is trace free by Lemma \ref{red-2}; thus taking the trace in the last displayed equation yields 
$$ \tr \widetilde{\Delta}_E(h_2-h_1^2)=2\di^{\star} \delta^g(h_2-h_1^2)
$$
by also using Lemma \ref{tr-00}, (ii). The operator $\tfrac{1}{2}\widetilde{\Delta}_Eh_1^2+Eh_1^2-\bfv(h_1,h_1)$ is trace free  according to Proposition \ref{tr-v}. Thus by applying the trace we find 
$$ \di^{\star} \left  ( \delta^g(h_2-h_1^2)+\tfrac{1}{4}\di\!\tr(h_1^2) \right )=0.
$$
Since we know that $\delta^g(h_2-h_1^2) $ is exact the previously displayed equation forces the vanishing of 
$\delta^g(h_2-h_1^2)+\tfrac{1}{4}\di\!\tr(h_1^2)$ and the claim is proved.
\end{proof}
This proves Theorem \ref{norm-thm} in the introduction.
\section{Negative K\"ahler-Einstein metrics}
\subsection{Elements of K\"ahler geometry}\label{prel-K}
Let $(M^{2m},g,J)$ be K\"ahler-Einstein, with Einstein constant $E<0$. We briefly review in this section some notation and main facts needed in what follows. The action of $J$ extends to $\Omega^{\star}M$ according to $J\alpha:=\alpha(J \cdot ,\ldots, J \cdot)$.
The bundle of symmetric 2-tensors splits as $\Sym^2 TM=\Sym^{2,+}TM \oplus \Sym^{2,-}TM$ where the summands
$$\Sym^{2,\pm }TM=\{h \in \Sym^2TM : hJ=\pm Jh \}.$$
This allows defining spaces of $\TT$-tensors according to 
$$\TT^{\pm}(M,g):=\TT(M,g) \cap \Gamma \left (\Sym^{2,\pm}TM \right ).$$ 
Furthermore, the operator $\delta^{\star_g}$ splits as 
$$ \delta^{\star_g}X=\delta^{\star_g,+}X+\delta^{\star_g,-}X
$$
according to $\Sym^2TM=\Sym^{2,+}TM \oplus \Sym^{2,-}TM $. The components are explicitly given by 
\begin{equation} \label{delta-K}
4\left (\delta^{\star_g,+}X \right )\circ J =(1+J)\di(JX)^{\flat} \ \ \mathrm{and} \ -2\delta^{\star_g,-}X=\L_{JX}J+\tfrac{1}{2}(1-J)\di\!X^{\flat}
\end{equation}
whenever $X \in TM$. Indeed, differentiating in $g(J \cdot ,\cdot)=\omega$ and using Cartan's formula leads to the identity 
$2g((\delta^{\star_g}X) \circ J, \cdot)+g(\L_{X}J \cdot ,\cdot)=\di (JX)^{\flat}$; the relations in \eqref{delta-K} follow now 
by projection onto $\Sym^{2,\pm}TM$.

In order to obtain the comparison formulas needed in what follows we first observe that 
\begin{lema} \label{div-sh}
Assume that $H \in \Sym^2TM$. Then 
\begin{equation*}
\delta^g\di_{\nabla^g}\!H(J \cdot ,J\cdot)=E[J,H]\circ J+(\ring{R}-E)H-\left (\nabla^g\delta^g(H \circ J) \right ) \circ J.
\end{equation*}
\end{lema}
\begin{proof}
With respect to some local orthonormal frame $\{e_i \}$ in $TM$ we have 
\begin{equation*}
-\left (\delta^g\di\!H(J \cdot ,J\cdot) \right )_X=((\nabla^g)^{2}_{e_i,Je_i}H)JX-((\nabla^g)^2_{e_i,JX}H)Je_i
\end{equation*}
since $\di_{\nabla^g}\!H(Je_i, JX)=(\nabla^g_{Je_i}H)JX-(\nabla^g_{JX}H)Je_i$ whenever $X \in TM$. Using the Ricci identity shows that 
$$((\nabla^{g})^2_{e_i,Je_i}H)JX=\tfrac{1}{2} \left ((\nabla^g)^{2}_{e_i,Je_i}H-(\nabla^g)^{2}_{Je_i,e_i}H \right )JX=-\tfrac{1}{2}[R^g(e_i,Je_i),H]JX=
-E[J,H]JX.$$ Using again the Ricci identity leads to 
\begin{equation*}
\begin{split}
-((\nabla^g)^2_{e_i,JX}H)Je_i=&-((\nabla^g)^2_{JX,e_i}H)Je_i+[R^g(e_i,JX),H]Je_i\\
=&\nabla^g_{JX}\delta^g(HJ)-[R^g(e_i,X),H]e_i=
\nabla^g_{JX}\delta^g(HJ)+(E-\ring{R}H)X
\end{split}
\end{equation*}
where in the second equality we have replaced the frame $\{e_i\}$ by $\{Je_i\}$. The claim follows now by collecting terms.
\end{proof}
Indicate with $\lambda^2M=\{\alpha \in \Lambda^{2}M : \alpha(J \cdot, J \cdot)=-\alpha\}$, so that $\Lambda^2M=\Lambda^{1,1}M \oplus \lambda^2M$. Accordingly, the exterior derivative $\di : \Omega^1M \to \Omega^2M$ splits into complex types according to $\di=\di^{+}+\di^{-}$ where the 
summands are given by 
$$ 2\di^{+}\! \alpha=\di\!\alpha+\di\!\alpha(J \cdot, J \cdot) \in \Omega^{1,1}M \ \mathrm{and} \ 2\di^{-}\! \alpha=\di\!\alpha-\di\!\alpha(J \cdot, J \cdot) \in \lambda^2M 
$$
whenever $\alpha$ belongs to $\Omega^1M$. Similarly, we consider the splitting of real bundles given by $\Lambda^{2}(M,TM)=\Lambda^{1,1}(M,TM) \oplus \lambda^2(M,TM)$ where in analogy with the case of $2$-forms we define 
$\lambda^2(M,TM):=\{\alpha \in \Lambda^{2}(M,TM) : \alpha(J \cdot, J \cdot)=-\alpha\}$. Accordingly, 
whenever $H$ belongs to $\Gamma \left (\Sym^{2,+}TM \right )$ we may split $\di_{\nabla^g}H=\di^{+}_{\nabla^g}H+\di^{-}_{\nabla^g}H$ where 
$$\di^{\pm}_{\nabla^g}H:=\tfrac{1}{2}\left (\di_{\nabla^g}H\pm\di_{\nabla^g}H(J \cdot, J \cdot)\right ). 
$$
Then $\di^{+}_{\nabla^g}H \in \Omega^{1,1}(M,TM)$ respectively $\di^{-}_{\nabla^g}H \in \lambda^2(M,TM)$.

As a first consequence of Lemma \ref{div-sh} we record that 
\begin{coro} \label{co-1}
Letting $H$ be in $ \Gamma \left (\Sym^{2,+}TM \right )$ we have 
\begin{equation*}
\delta^g\di^{+}_{\nabla^g}H=(\tfrac{1}{2}\Delta_E+\ring{R})H-\delta^{\star_g,-}\delta^gH-\tfrac{1}{2}\di^{-} \delta^gH.
\end{equation*}
\end{coro}
\begin{proof}
By Lemma \ref{div-sh} we get 
$$\delta^g\di_{\nabla^g}\!H(J \cdot ,J\cdot)=(\ring{R}-E)H-\nabla^g(\delta^g(HJ)) \circ J$$
since $HJ=JH$. To deal with the divergence term we start from observing that the equality $\delta^g(HJ)=J\delta^gH$ entails 
$\nabla^g(\delta^g(HJ)) \circ J=J \left (\nabla^g(\delta^gH)\right )\circ J$. 
After expanding 
\begin{equation} \label{expn-N}
\nabla^g=\delta^{\star_g,+}+\delta^{\star_g,-}+\tfrac{1}{2}(\di^{+} +\di^{-})
\end{equation}
 we find 
$J \left (\nabla^g(\delta^gH)\right )\circ J=\left (-\delta^{\star_g,+}+\delta^{\star_g,-}+\tfrac{1}{2}(-\di^{+} +\di^{-}) \right ) \delta^gH$. Taking this fact into account shows that 
$\delta^g \di_{\nabla^g}H(J \cdot ,J \cdot)=(\ring{R}-E)H+\left (\delta^{\star_g,+}-\delta^{\star_g,-}+\tfrac{1}{2}(\di^{+} -\di^{-}) \right ) \delta^gH.$
The claim follows now after re-writing the Weitzenb\"ock formula in \eqref{wz1} as 
$$\delta^g \di_{\nabla^g}H=-\left (\delta^{\star_g,+}+\delta^{\star_g,-}+\tfrac{1}{2}(\di^{+} +\di^{-}) \right ) \delta^gH+(\Delta_E+\ring{R}+E)H.$$
\end{proof}
We also recall below the comparison formula between the Einstein operator acting on $\Gamma \left (\Sym^{2,+}TM \right )$ and the Hodge Laplacian $\Delta^g$ acting on $\Omega^{1,1}M$. This formula reads 
\begin{equation} \label{lap-comp2}
 g^{-1}(\Delta^g-2E)A=(\Delta_E H ) \circ J
\end{equation}
whenever $H \in \Gamma \left (\Sym^{2,+}TM \right )$, where $A:=g(HJ \cdot ,\cdot)$ belongs to $\Omega^{1,1}M$. In particular this entails 
$$ \ker \Delta_E \cap \Gamma \left (\Sym^{2,+}TM \right )=\{0\}
$$ 
when $E<0$. 

Whenever $h$ belongs to $ \Gamma \left (\Sym^{2,-}TM \right )$ we define the operator 
$$\bdel h:=\tfrac{1}{2}\left ( \di_{\nabla^g}h-\di_{\nabla^g}h(J \cdot, J \cdot)\right )$$ 
and also record that $[h,J]^{\FN}(J \cdot ,\cdot)=\bdel h$; equivalently, $[h,J]^{\FN}=\bdel(Jh)$. We also let $\widetilde{\Delta}_E^{-}$ be the component on $\Sym^{2,-}TM$ of the restriction of the perturbation $\widetilde{\Delta}_E$ of the Einstein operator to $\Sym^{2,-}TM$; since elements of the latter space are trace-free we obtain 
$$ \widetilde{\Delta}_E^{-}h=\Delta_Eh-2\delta^{\star_g,-}\delta^gh
$$
for all $h$ in $\Gamma \left (\Sym^{2,-}TM \right )$.
\begin{coro} \label{bde}
Assume that $h \in \Gamma \left (\Sym^{2,-}TM \right )$. Then 
$$ \delta^g \bdel h=\tfrac{1}{2} \widetilde{\Delta}_E^{-}h-\tfrac{1}{2}\di^{-} \delta^g h.
$$
\end{coro}
\begin{proof}
Follows from Lemma \ref{div-sh} in complete analogy with the proof of Corollary \ref{co-1}.
\end{proof}
In particular, as showed in \cite{Koiso-I} we have 
\begin{equation*}
\ker \Delta_E \cap \Gamma \left (\Sym^{2,-}TM \right )=\{h \in \Gamma \left (\Sym^{2,-}TM \right ) : \bdel h=0 \ \mathrm{and} \ \delta^gh=0\}.
\end{equation*}
Summarising, we thus recall that for K\"ahler-Einstein metrics with $E<0$ the space of infinitesimal Einstein deformations is given by  
\begin{equation*}
\mathscr{E}(M,g)=\{h \in \Gamma \left (\Sym^{2,-}TM \right ) : \bdel h=0 \ \mathrm{and} \ \delta^g h=0\}.
\end{equation*}

\subsection{Type decomposition of $[\cdot,\cdot]^{\FN}$ and the Kodaira-Spencer bracket} \label{cpx-bra}
Letting $h$ belong to $\Sym^{2,-}TM$ the bracket $[h,h]^{\FN}$ lives in $\Lambda^{2}(M,TM)$; by considering the type decomposition 
of the latter bundle we will explain how this bracket compares to the Kodaira-Spencer bracket by using pure representation theoretical considerations. We recall (see section \ref{prel-K}) that the bundle
$\Lambda^{2}(M,TM)$ splits into complex types as 
\begin{equation} \label{split1}
\Lambda^{2}(M,TM)=\Lambda^{1,1}(M,TM) \oplus \lambda^2(M,TM)
\end{equation}
where $\lambda^2(M,TM)=\{\alpha \in \Lambda^{2}(M,TM) : \alpha(J \cdot, J \cdot)=-\alpha\}$. The latter bundle splits further as 
\begin{equation} \label{split2}
\begin{split}
\lambda^2(M,TM)=\lambda^2_{+}(M,TM)\oplus \lambda^2_{-}(M,TM)
\end{split}
\end{equation}
where 
\begin{equation*} 
\begin{split}
&\lambda^2_{-}(M,TM):= \{\gamma \in \lambda^2(M,TM): \gamma(X,JY)=-J\gamma(X,Y)\} \mathrm{and} \\
&\lambda^2_{+}(M,TM):= \{\gamma \in \lambda^2(M,TM): \gamma(X,JY)=J\gamma(X,Y)\}.
\end{split}
\end{equation*}
\begin{rema} \label{+rmk}
The space $\lambda^2_{+}(M,TM)$ is canonically isomorphic to $\Lambda^{(1,2)+(2,1)}M$. The inverse of this 
isomorphism is given by $T \in \Lambda^{(1,2)+(2,1)}M \mapsto \tfrac{1}{2}\left (T -T(J \cdot, J\cdot) \right )$. 
\end{rema}
For subsequent use we record that if $\gamma$ belongs to $\lambda^2(M,TM)$ the projectors are given by 
$$ 2\gamma_{-}(X,Y)=\gamma(X,Y)+J\gamma(X,JY) \ \mathrm{and} \ 2\gamma_{+}(X,Y)=\gamma(X,Y)-J\gamma(X,JY).
$$
Notice that $\bdel h$ belongs to $\lambda^2_{-}(M,TM)$ whenever $h \in \Sym^{2,-}TM$. In addition, indicating with  $\mathrm{a} : \Lambda^2(M,TM) \to \
\Lambda^3M$ the total antisymmetrisation map, that is 
$$\mathrm{a}(\alpha)(X,Y,Z)=\mathfrak{S}_{X,Y,Z}g(\alpha(X,Y),Z)$$ 
we must have $\mathrm{a}(\bdel h)=0$.

Splitting the Fr\"olicher-Nijenhuis bracket according to 
\eqref{split1} yields 
$$ [h,h]^{\FN}=[h,h]^{\FN,+}+[h,h]^{\FN,-}
$$
where the components 
\begin{equation*}
\begin{split}
&2[h,h]^{\FN,+}(X,Y)=[h,h](X,Y)+[h,h](JX,JY) \ \mathrm{and} \\
&2[h,h]^{\FN,-}(X,Y)=[h,h](X,Y)-[h,h](JX,JY).
\end{split}
\end{equation*}
Define now $\partial^g h:=\tfrac{1}{2}\left ( \di_{\nabla^g}h+ \di_{\nabla^g}h(J \cdot,J \cdot) \right )$ so that 
$ \di_{\nabla^g}h=\partial^g h+\bdel h$. Then $h \sharp \partial^g h$ belongs to $\lambda^2(M,TM)$ hence 
\begin{equation*}
\begin{split}
&[h,h]^{\FN,+}=-h \sharp \bdel h+\di^{+}_{\nabla^g}h^2 \ \mathrm{and} \ [h,h]^{\FN,-}=-h \sharp \partial ^g h+\di^{-}_{\nabla^g}h^2
\end{split}
\end{equation*}
by using the comparison formula \eqref{bra-nac}.
\begin{lema} \label{id-plus}
Assume that $h \in \Gamma \left ( \Sym^{2,-}TM \right )$. Then 
$$ (\nabla^g_{X}h)Y+(\nabla^g_{JX}h)JY=\partial^g h(X,Y)+\partial^g (Jh)(X,JY)
$$
for all $X,Y $ in $TM$.
\end{lema}
\begin{proof}
Start from the identity 
$ (\nabla^g_{X}h)Y+(\nabla^g_{JX}h)JY=(\di_{\nabla^g}h)(X,Y)+(\di_{\nabla^g}hJ)(JX,Y)
$
whenever $X,Y \in TM$. Having $\bdel h$ in $\lambda^2_{-}(M,TM)$ entails 
$$\bdel (hJ)(JX,Y)=\bdel(hJ)(X,JY)=-J\bdel (hJ)(X,Y)=
-\bdel h(X,Y).$$
At the same time, $\partial^g (hJ)(JX,Y)=-\partial^g (hJ)(X,JY)=
\partial^g (Jh)(X,JY)$ since $\partial^g (hJ)$ belongs to $\Omega^{1,1}(M,TM)$ and the claim follows.
\end{proof}
Based on this we establish the following 
\begin{lema} \label{type-bra}
Assume that $h \in \Gamma \left ( \Sym^{2,-}TM \right )$. The components of $h \sharp \partial^g h$ with respect to 
\eqref{split2} read 
\begin{equation*}
\begin{split}
&2(h \sharp \partial^g h)_{-}=h \sharp \partial^g h-Jh \sharp \partial^g Jh\\
&2(h \sharp \partial^g h)_{+}=\di_{\nabla^g}\!h^2-\di_{\nabla^g}\!h^2(J \cdot, J\cdot)-2h \circ \bdel h.
\end{split}
\end{equation*}
\end{lema}
\begin{proof}
We have 
$2(h \sharp \partial^g h)_{-}(X,Y)=(h \sharp \partial^g h)(X,Y)+J (h \sharp \partial^g h)(X,JY).$ Since 
$\partial^g h$ belongs to $\Omega^{1,1}(M,TM)$ we have $J \partial^g h(hX,JY)=-J\partial^gh(JhX,Y)=-\partial^g(Jh)(JhX,Y)$; a straightforward computation based on expanding the term $J (h \sharp \partial^g h)(X,JY)$ then yields the equality 
$2(h \sharp \partial^g h)_{-}=h \sharp \partial^g h-Jh \sharp \partial^g Jh$.

To determine the component $(h \sharp \partial^g h)_{+}$ we start from 
\begin{equation*}
\begin{split}
2(h \sharp \partial^g h)_{+}(X,Y)=&(h \sharp \partial^g h)(X,Y)-J (h \sharp \partial^g h)(X,JY)\\
=&\partial^gh(hX,Y)+\partial^gh(X,hY)-J \left (\partial^gh(hX,JY)+\partial^gh(X,hJY) \right ).
\end{split}
\end{equation*}
Because $\partial^g h$ belongs to $\Omega^{1,1}(M,TM)$ we have $\partial^gh(hX,JY)=-\partial^g h(JhX,Y)=\partial^g h(hJX,Y)$; thus 
$$ 2(h \sharp \partial^g h)_{+}=h \sharp \partial^g h+Jh \sharp \partial^g Jh.
$$
Operating the variable change $Y \mapsto hY$ in Lemma \ref{id-plus} yields 
$$(\nabla^g_{X}h)hY-(\nabla^g_{JX}h)hJY=\partial^g h(X,hY)+\partial^g (Jh)(X,JhY).
$$
Re-writing $(\nabla^g_{X}h)hY=(\nabla_X^gh^2)Y-h (\nabla_X^gh)Y$ and anti-symmetrising in the variables $(X,Y)$ thus leads to 
$$\di_{\nabla^g}\!h^2-\di_{\nabla^g}\!h^2(J \cdot, J\cdot)-2h \circ \bdel h=h \sharp \partial^g h+Jh \sharp \partial^g Jh
$$
and the claim is proved.
\end{proof}
\begin{defi} \label{K-S}
We define the Kodaira-Spencer bracket according to 
$$ [h,h]^c:=[h,h]^{\FN,-}_{-} \ \mathrm{in} \ \lambda^2_{-}(M,TM)
$$
whenever $h \in \Sym^{2,-}TM$. 
\end{defi}
The first main fact we establish in this section is that the Kodaira-Spencer bracket compares to the Fr\"olicher-Nijenhuis bracket as indicated below.
\begin{pro} \label{com-brak}
Assume that $h$ belongs to $\Gamma \left ( \Sym^{2,-}TM \right )$. The following hold
\begin{itemize}
\item[(i)] the full type decomposition of $[h,h]^{\FN}$ reads 
\begin{equation} \label{comp-1}
[h,h]^{\FN}=[h,h]^c+h \circ \bdel h+\left ( \di_{\nabla^g}^{+}h^2-h \sharp \bdel h \right )
\end{equation}
according to $\Lambda^2(M,TM)=\lambda^{2}_{-}(M,TM) \oplus \lambda^{2}_{+}(M,TM) \oplus \Lambda^{1,1}(M,TM)$
\item[(ii)] we have 
\begin{equation} \label{comp-0}
\begin{split}
[h,h]^{c}=&-\tfrac{1}{2} \left (h \sharp \partial^g h-Jh \sharp \partial^g( Jh) \right )
=\tfrac{1}{2}\left ([h,h]^{\FN}-[Jh,Jh]^{\FN} \right ).
\end{split}
\end{equation}
\end{itemize}
\end{pro}
\begin{proof}
(i) We have 
\begin{equation*}
[h,h]^{\FN}=[h,h]^{\FN,+}+[h,h]^{\FN,-}=-h \sharp \bdel h+\di^{+}_{\nabla^g}h^2-h \sharp \partial^g h+\di^{-}_{\nabla^g}h^2
\end{equation*}
where, to shorten notation, we have set $\di^{-}_{\nabla^g}h^2:=\tfrac{1}{2} \left (\di_{\nabla^g}\!h^2-\di_{\nabla^g}\!h^2(J \cdot, J\cdot) \right )$. Furthermore, using Lemma \ref{type-bra} reveals that 
$$h \sharp \partial^g h=(h \sharp \partial^g h)_{+}+(h \sharp \partial^g h)_{-}=\di_{\nabla^g}^{-}h^2-h \circ \bdel h+(h \sharp \partial^g h)_{-}.
$$
After gathering terms we end up with 
$$ [h,h]^{\FN}=-(h \sharp \partial^g h)_{-}+h \circ \bdel h+\di_{\nabla^g}^{+}h^2-h \sharp \bdel h.
$$
The type of the summands above is given by $(h \sharp \partial^g h)_{-}$ in $ \lambda^2_{-}(M,TM)$ as well as 
$h \circ \bdel h$ in $\lambda^2_{+}(M,TM)$ and $\di_{\nabla^g}^{+}h^2-h \sharp \bdel h \in \Omega^{1,1}(M,TM)$. This proves \eqref{comp-1} and also that 
$$ [h,h]^c=-(h \sharp \partial^g h)_{-}.
$$
(ii) Using the equation right above and Lemma \ref{type-bra} shows that 
$$[h,h]^{c}=-\tfrac{1}{2} \left (h \sharp \partial^g h-Jh \sharp \partial^g( Jh) \right ).$$
To prove the second half of the claim we first record that 
having $\bdel h \in \lambda^2_{-}(M,TM)$ entails 
$$\bdel(Jh)(JhX,Y)=\bdel(Jh)(hX,JY)=-J\bdel(Jh)(hX,Y)=\bdel h(hX,Y).
$$
Similarly, $\bdel(Jh)(X,JhY)=-J\bdel(Jh)(X,hY)=\bdel h(X,hY). $ These facts entail that 
$$ h \sharp \bdel h=Jh  \sharp \bdel (Jh).
$$
Using \eqref{bra-nac} thus leads to $ [h,h]^{\FN}-[Jh,Jh]^{\FN}=-\left (h \sharp \partial^g h-Jh \sharp \partial^g Jh \right ).
$ The claim is now fully proved.
\end{proof}
The type decomposition in \eqref{comp-1} will be used in section \ref{ref-str} in order to express the restriction of $\bfv $ to 
$\Gamma \left (\Sym^{2,-}TM \right )$ in terms of the Kodaira-Spencer bracket.
\begin{rema} \label{del-b}
\begin{itemize}
\item[(i)] Classically, the Kodaira-Spencer bracket arises in relation to the integrability of almost complex structures of the form
$J_h:=(1-h)^{-1}J(1-h)$ where $h \in \Gamma \left (\End^{-}TM \right )$. Indeed, $J_h$ is complex integrable if and only if $h$ satisfies the Maurer-Cartan equation $\bdel h+[h,h]^c=0$; see \cite{Huy}[Lemma 6.1.2].
\item[(ii)] The Kodaira-Spencer bracket is preserved by the operator
$\overline{\partial}$. In particular, if $h$ in $\Gamma(\Sym^{2,-}TM)$ satisfies $\bdel h=0$ then $\overline{\partial}[h,h]^c=0$.(see e.g. \cite{Huy}[Remark 6.1.1])
\end{itemize}
\end{rema}
Below we describe explicitly how to perturb the Kodaira-Spencer bracket in order to obtain, after taking the divergence, a symmetric and divergence-free tensor. 
\begin{defi}\label{pert-bra}
We define the perturbed Kodaira-Spencer bracket according to 
\begin{equation*} 
[h,h]_{\TT}^c:=[h,h]^c+\tfrac{1}{2}\left (\delta^gh \wedge h-\delta^g(Jh) \wedge Jh \right ) \ \mathrm{in} \ \lambda^{2}_{-}(M,TM)
\end{equation*}
whenever $h$ belongs to $\Sym^{2,-}TM$. 
\end{defi}
Note that a quick algebraic computation shows that 
$\delta^gh \wedge h-\delta^g(Jh) \wedge Jh $ belongs to $\lambda^{2}_{-}(M,TM)$ hence so does the perturbed complex bracket 
$[h,h]^c_{\TT}$. This bracket extends by polarisation to a symmetric bracket on the space $\Gamma \left (\Sym^{2,-}TM \right )$ via the usual formula $2[h_1,h_2]_{\TT}^c=
[h_1+h_2,h_1+h_2]_{\TT}^c-[h_1,h_1]_{\TT}^c-[h_2,h_2]_{\TT}^c$. The main properties of this new bracket are summarised below.
\begin{pro} \label{del-TT}Assuming that $h$ is in $\Gamma \left (\Sym^{2,-}TM \right )$ the following hold
\begin{itemize}
\item[(i)] the divergence 
$$ \delta^g [h,h]^c_{\TT} \ \mathrm{belongs \ to} \ \TT^{-}(M,g)
$$
\item[(ii)] we have  $\mathrm{a} [h,h]^c=0$ and $\mathrm{a} [h,h]_{\TT}^c=0$.
\end{itemize}
\end{pro}
\begin{proof}
(i) Recall that $[h,h]^{c}
=\tfrac{1}{2}\left ([h,h]^{\FN}-[Jh,Jh]^{\FN} \right )$ by \eqref{comp-0}. Due to the algebraic identities $Jh \circ \di_{\nabla^g}(Jh)=h \circ \di_{\nabla^g}h$ and $(\ring{R}(Jh)) \circ Jh=\ring{R}(h) \circ h$ applying successively Proposition \ref{sym-branN} for $h$ respectively $Jh$ ensures that the tensor $\delta^g[h,h]^c_{\TT}$ is symmetric. Because 
$[h,h]^c_{\TT}$ lies in $\lambda^2_{-}(M,TM)$ its divergence is $J$-anti-invariant thus 
$\delta^g[h,h]^c_{\TT}\in \Sym^{2,-}TM$. To show it is divergence-free we compute 
\begin{equation*}
\begin{split}
\la \delta^g[h,h]^c_{\TT}, \delta^{\star_g}X \ra_{L^2}=&\la \delta^g[h,h]^c_{\TT}, \delta^{\star_g,-}X \ra_{L^2}
=-\tfrac{1}{2}
\la \delta^g [h,h]^c_{\TT},\L_{JX}J+\di^{-}\!X^{\flat} \ra_{L^2}\\
=&-\tfrac{1}{2}
\la \delta^g[h,h]^c_{\TT},\L_{JX}J \ra_{L^2}
\end{split}
\end{equation*}
since $\delta^g[h,h]^c_{\TT}$ is symmetric and $J$-anti-invariant. However, by $L^2$-duality and after taking again into account that 
$[h,h]^c_{\TT}$ is $J$-anti-invariant we obtain 
\begin{equation*}
\begin{split}
\la \delta^g[h,h]^c_{\TT},\L_{JX}J \ra_{L^2}&=
\la [h,h]^c_{\TT}, \di_{\nabla^g}(\L_{JX}J) \ra_{L^2}
=\la [h,h]^c_{\TT}, \bdel(\L_{JX}J) \ra_{L^2}=0
\end{split}
\end{equation*}
and the proof is finished, since $\bdel(\L_{JX}J)=0$.\\
(ii) Since $h$ is symmetric, a direct computation based on \eqref{bra-na} shows that the total alternation $\mathrm{a}[h,h]^{\FN}=
\mathrm{a}(h \circ \di_{\nabla^g}h)$. That $\mathrm{a} [h,h]^c$ thus follows from the identity $h \circ \di_{\nabla^g}h=
Jh \circ \di_{\nabla^g}Jh$ combined with the expression for $[h,h]^c$ given in \eqref{comp-0}. A direct algebraic computation based on having $h$ and $hJ$ symmetric shows that $\mathrm{a}(\delta^gh \wedge h-\delta^g(Jh) \wedge Jh)=0$ hence $\mathrm{a} [h,h]^c_{\TT}=0$ 
as well.\\
\end{proof}
For divergence-free tensors we have the following immediate consequence of Proposition \ref{del-TT} above.
\begin{pro}\label{sym-N}
Assuming that $h$ in $\Gamma \left (\Sym^{2,-}TM \right )$ satisfies $\delta^gh=0$ we have that
$$ \delta^g[h,h]^c \ \mathrm{belongs \ to } \ \TT^{-}(M,g).$$
\end{pro}
To conclude this section we also compare the divergences of the Fr\"olicher-Nijenhuis respectively 
Kodaira-Spencer brackets with respect to the type decomposition of the bundle of endomorphisms of $TM$. 
\begin{pro} \label{sym-cpx}
Assume that $h \in \Gamma \left ( \Sym^{2,-}TM \right )$ satisfies $\bdel h=0$ and $\delta^g h=0$. With respect to the type decomposition 
$$\End(TM)=\Sym^2TM \oplus \Lambda^2M=\Sym^{2,+}TM \oplus \Sym^{2,-}TM \oplus \Lambda^2M $$
the tensor $\delta^g[h,h]^{\FN}$ splits as 
\begin{equation*}
\delta^g[h,h]^{\FN}=(\tfrac{1}{2}\Delta_E+\ring{R})h^2+\left (\delta^g[h,h]^{c}-\delta^{\star_g,-}\delta^g h^2 \right )-\tfrac{1}{2} \d^{-} \delta^g h^2.
\end{equation*}
\end{pro}
\begin{proof}
Because $[h,h]^c \in \lambda^2_{-}(M,TM)$ we have $\delta^g[h,h]^c \in \Gamma \left (\Sym^{2,-}TM \right )$. At the same time we have 
$[h,h]^{\FN}=[h,h]^c +\di_{\nabla^g}^{+}h^2$ by \eqref{comp-1}. The claim follows now from Corollary \ref{co-1}.
\end{proof}
\subsection{Structure of the operator $\bfv$} \label{ref-str}
The aim in this section is to use the previous considerations in order to show that second order infinitesimal Einstein deformations can be explicitely 
determined. To this aim we need to render explicit the structure of the operator $\bfv$ with respect to the type splitting of symmetric tensors. The first observation in this direction is that the restriction of the operator $\bfv$ to $\Sym^{2,-}TM$ is entirely determined 
by the Kodaira-Spencer bracket.
\begin{pro}\label{3-}
Assume that $h_i$ belong to $ \Gamma \left (\Sym^{2,-}TM \right )$ for $1 \leq i \leq 3$. Then 
$$ \bfv(h_1,h_2,h_3)=2 \mathfrak{S}_{abc} \la [h_a,h_b]^c, \bdel h_c\ra_{L^2}-\tfrac{1}{2} \mathfrak{S}_{abc} \la \delta^{g}\{h_a,h_b \}, \delta^gh_c\ra_{L^2}
$$
with cyclic permutations on the indices $abc$.
\end{pro}
\begin{proof}
According to \eqref{comp-1} we have 
$[h,h]^{\FN}=[h,h]^c+h \circ \bdel h+\left ( \di_{\nabla^g}^{+}h^2-h \sharp \bdel h \right )$
whenever $h \in \Sym^{2,-}TM$. Since $h \circ \bdel h$ belongs to $\lambda^{2}_{+}(M,TM)$ its divergence belongs to $\End^{+}TM$; thus 
\begin{equation*}
\la \delta^g[h,h]^{\FN},h\ra_{L^2}=\la [h,h]^c, \di_{\nabla^g}\!h\ra_{L^2}+\la \delta^g\di^{+}_{\nabla^g}h^2,h \ra_{L^2}-\la h \sharp \bdel h, \di_{\nabla^g}\!h \ra_{L^2}.
\end{equation*}
Furthermore, 
\begin{equation*}
\begin{split}
\la h \sharp \bdel h, \di_{\nabla^g}\!h \ra_{L^2}=&\la \bdel h, h \sharp \di_{\nabla^g}\!h \ra_{L^2}=\la \bdel h,\di_{\nabla^g}\!h^2-[h,h]^{\FN} \ra_{L^2}\\
=&\la \delta^g \bdel h, h^2 \ra_{L^2}-\la \bdel h, [h,h]^c+h \circ \bdel h \ra_{L^2}=\la \delta^g \bdel h, h^2 \ra_{L^2}-\la \bdel h, [h,h]^c \ra_{L^2}
\end{split}
\end{equation*}
since $h \circ \bdel h$ belongs to $\lambda^2_{+}(M,TM)$. It follows that 
$$ \la \delta^g[h,h]^{\FN},h\ra_{L^2}=2\la [h,h]^c,\bdel h \ra_{L^2}+\la \delta^g \di^{+}_{\nabla^g}h^2, h \ra_{L^2}-\la \delta^g \bdel h,h^2 \ra_{L^2}.
$$
Now we use the comparison formula in Corollary \ref{co-1} which yields 
$$\delta^g \di^{+}_{\nabla^g}h^2=(\tfrac{1}{2}\Delta_E+\ring{R})h^2-\delta^{\star_g,-}\delta^gh^2
-\tfrac{1}{2}\di^{-} \delta^gh^2.$$
Because $h$ is symmetric and $J$-anti-invariant, after taking into account that $\Delta_E$ and $\ring{R}$ preserve complex type we find that $\la \delta^g \di^{+}_{\nabla^g}h^2, h \ra_{L^2}=-\la \delta^gh^2,\delta^gh\ra_{L^2}$. Similar type arguments based this time on 
Corollary \ref{bde} show that $\la \delta^g \bdel h,h^2 \ra_{L^2}=0$. We thus get 
$$\tfrac{1}{3}\bfv(h,h,h)=\la \delta^g[h,h]^{\FN},h\ra_{L^2}=2\la [h,h]^c,\bdel h \ra_{L^2}-\la \delta^g h^2,\delta^g h\ra_{L^2}$$
and the claim follows now by polarisation. 
\end{proof}

Our next objective is to explicitly compute quantities of the type 
$$ \la \bfv(h,h), H \ra_{L^2}
$$
where $h \in \Gamma \left (\Sym^{2,-}TM \right )$ satisfies $\overline{\partial}h=0$ as well as $\delta^gh=0$ and the tensor 
$H \in \Gamma \left (\Sym^{2,+}TM \right )$. This is considerably more technically involved than Proposition \ref{3-} and requires a series of preliminary lemmas.
\begin{lema} \label{L-l}
Assume that $h \in \Sym^{2,-}TM$ and 
consider the symmetric operator given by $\mathrm{L}(h):=\nabla_{e_i}^gh \circ \nabla_{e_i}^gh$. Then 
$$ \mathrm{L}(h)=\tfrac{1}{2}\{(\Delta_E+2\ring{R})h,h\}-\tfrac{1}{2}(\Delta_E+2\ring{R})h^2.
$$
\end{lema}
\begin{proof}
A straightforward computation shows that $2\mathrm{L}(h)=\{\nabla^{\star_g}\nabla^g h,h\}-\nabla^{\star_g}\nabla^g h^2$. To conclude, we use that $\Delta_E+2\ring{R}=\nabla^{\star_g}\nabla^g$.
\end{proof}
The next observation deals with the divergence type expression in the Lemma \ref{div-11} below. In the computations in the rest of this section we will 
repeatedly use the following fact which is entailed by the definition of $\bdel$. Namely, if $h \in \Gamma \left (\Sym^{2,-}TM \right )$ then the tensor 
\begin{equation} \label{sym-h}
(X,Y) \mapsto  (\nabla^g_{JX}h)JY-(\nabla^g_Xh)Y+\bdel h(X,Y)\ \mathrm{is \ symmetric}.
\end{equation}
\begin{lema} \label{div-11}
Assume that $h \in \Gamma \left (\Sym^{2,-}TM \right )$. We have
$$ -\la 
\nabla^g_{\delta^g(HJ)}h,hJ \ra_{L^2}=\la \delta^{\star_g} \left ((2\delta^g+\tfrac{1}{2}\di \tr )h^2-2h \delta^gh \right ), H \ra_{L^2}
-2 \la \bdel h, \delta^gH \wedge h \ra_{L^2}$$
for all $H \in \Gamma \left (\Sym^{2,+}TM \right )$.
\end{lema}
\begin{proof}
Letting $X:=\delta^gH$ we have $\delta^g(HJ)=JX$ and $-\la 
\nabla^g_{\delta^g(HJ)}h,hJ \ra_{L^2}=\la 
\nabla^g_{JX}(hJ),h \ra_{L^2}$. We compute 
\begin{equation*}
\begin{split}
\la \nabla^g_{JX}(hJ),h \ra_{L^2}=&\la (\nabla^g_{JX}h)Je_i,he_i \ra_{L^2}\\
=&\la (\nabla^g_{JX}h)Je_i-(\nabla^g_{X}h)e_i,he_i \ra_{L^2}
+\la (\nabla^g_{X}h)e_i,he_i \ra_{L^2}\\
=&\la (\nabla^g_{Je_i}h)JX-(\nabla^g_{e_i}h)X,he_i \ra_{L^2}+2\la \bdel h(e_i,X),he_i \ra_{L^2}+\tfrac{1}{2}\L_X \tr(h^2)
\end{split}
\end{equation*}
by essentially taking into account the symmetry property in \eqref{sym-h}. Since $h$ is symmetric and $J$-anti-invariant we have 
$\la (\nabla^g_{Je_i}h)JX-(\nabla^g_{e_i}h)X,he_i \ra_{L^2}=
-2\la X, (\nabla^g_{e_i}h)he_i \ra_{L^2}.$ Furthermore $-(\nabla^g_{e_i}h)he_i=-(\nabla^g_{e_i}h^2)e_i+h(\nabla^g_{e_i}h)e_i=\delta^gh^2-h\delta^gh$. Summarising, we have showed that 
$\la \nabla^g_{JX}(hJ),h \ra_{L^2}=\la X,2\delta^gh^2-2h\delta^gh+\tfrac{1}{2}\di\!\tr h^2 \ra_{L^2}-2 \la \bdel h, \delta^gH \wedge h \ra_{L^2}$ after also taking into that $\la \bdel h, \delta^gH \wedge h \ra_{L^2}=-\la \bdel h(e_i,X),he_i \ra_{L^2}$.
The claim follows now by integration by parts.
\end{proof}
Yet another preliminary fact which will be needed is 
\begin{lema} \label{pl-4}
Letting $h$ belong to $\Gamma \left (\Sym^{2,-}TM \right )$ we have 
$$ \la (\nabla^g_{He_j}h)e_i,(\nabla^g_{Je_j}h)Je_i \ra_{L^2}=-\la \nabla^g_{\delta^g(HJ)}h,hJ\ra_{L^2}+2g(\ring{R}(h^2),H)
$$
whenever $H$ belongs to $\Gamma \left (\Sym^{2,+}TM \right )$.
\end{lema}
\begin{proof}
We first observe that 
$$ g\left ((\nabla^g_{He_j}h)e_i,(\nabla^g_{Je_j}h)Je_i \right )=-g(\nabla^g_{Ae_j}h)e_i,(\nabla^g_{e_j}h)Je_i)=-g(\nabla^g_{Ae_j}h,\nabla^g_{e_j}(hJ))$$ where the skew symmetric tensor 
$A:=HJ$. The rest of the proof is by integration by parts as follows. Consider the $1$-form $\alpha$ on $M$ given by $\alpha(X)=g(\nabla_{AX}^gh,hJ)$; its divergence reads 
\begin{equation*}
-\di^{\star_g}\alpha=g((\nabla^g)^2_{e_i,Ae_i}h,hJ)-g(\nabla^g_{\delta^gA}h,hJ)+g(\nabla^g_{Ae_j}h,\nabla^g_{e_j}(hJ)).
\end{equation*}
In addition, since $A$ is skew-symmetric, using the Ricci identity shows that 
\begin{equation*}
\begin{split}
g\left ((\nabla^g)^2_{e_i,Ae_i}h,hJ \right )=&\tfrac{1}{2}g((\nabla^g)^2_{e_i,Ae_i}h-(\nabla^g)^2_{Ae_i,e_i}h,hJ)\\
=&-\tfrac{1}{2}g([R^g(e_i,Ae_i),h],hJ)=2g(\ring{R}H,h^2)=2g(\ring{R}(h^2),H),
\end{split}
\end{equation*}
after also taking into account that $R^g(e_i,Ae_i)=2J\ring{R}H$ and that $\ring{R}$ is a symmetric operator.
Integration over $M$ then yields  
$\la \nabla^g_{Ae_j}h,\nabla^g_{e_j}(hJ) \ra_{L^2}=\la \nabla^g_{\delta^gA}h,hJ\ra_{L^2}-2g(\ring{R}(h^2),H)$
and the claim follows.
\end{proof}
One last technical step which is subsequently needed  is contained in the following
\begin{lema} \label{bdd}
Assume that the pair $(h,H)$ belongs to $\Gamma \left (\Sym^{2,-}TM \right )\oplus \Gamma \left (\Sym^{2,+}TM \right )$. Then 
$$ \la \bdel h(e_i,He_j),(\nabla^g_{e_j}h)e_i \ra_{L^2}=-\la H \sharp \bdel h, \bdel h\ra_{L^2}+\la \tfrac{1}{4}\{h,\widetilde{\Delta}^{-}_Eh\}-\tfrac{1}{4}
[h,\di^{-} \delta^gh]-\delta^g(h \circ \bdel h), H \ra_{L^2}.$$
\end{lema}
\begin{proof}
Start from $\la \bdel h(e_i,He_j),(\nabla^g_{e_j}h)e_i) \ra_{L^2}=
\la \bdel h(e_i,He_j), \di_{\nabla^g}h(e_j,e_i)+(\nabla^g_{e_i}h)e_j) \ra_{L^2}$. The first summand reads 
\begin{equation*}
\begin{split}
\la \bdel h(e_i,He_j), \di_{\nabla^g}h(e_j,e_i) \ra_{L^2}=&-\tfrac{1}{2} \la (H \sharp \bdel h)(e_i,e_j), \di_{\nabla^g}h(e_j,e_i)\ra_{L^2}\\
=&-\la H \sharp \bdel h, \di_{\nabla^g}h\ra_{L^2}=-\la H \sharp \bdel h, \bdel h\ra_{L^2}.
\end{split}
\end{equation*}
In the last step we have used that the tensor $H \sharp \bdel h$ belongs to $\lambda^2(M,TM)$, as entailed by having $HJ=JH$.
Since $H$ is symmetric, 
\begin{equation*}
\begin{split}
\la \bdel h(e_i,He_j), (\nabla^g_{e_i}h)e_j \ra_{L^2}=&\la \bdel h(e_i,e_j), (\nabla^g_{e_i}h)He_j \ra_{L^2}\\
=&\la \bdel h(e_i,e_j), \nabla^g_{e_i}(h \circ H)e_j-h (\nabla^g_{e_i}H)e_j\ra_{L^2}\\
=&\la \bdel h(e_i,e_j), \nabla^g_{e_i}(h \circ H)e_j \ra_{L^2}-\la h \circ \bdel h(e_i,e_j), (\nabla^g_{e_i}H)e_j\ra_{L^2}\\
=& \la \delta^g \bdel h,h \circ H \ra_{L^2}-\la h \circ \bdel h, \di\!_{\nabla^g}H \ra_{L^2}
\end{split}
\end{equation*}
after also integrating by parts. In addition, by using the Weitzenb\"ock-type formula in Corollary \ref{bde} and that $h \circ H=\tfrac{1}{2}\{h,H\}+\tfrac{1}{2}[h,H]$ we obtain that 
\begin{equation*}
\begin{split}
\la \delta^g \bdel h,h \circ H \ra_{L^2}=&\la \tfrac{1}{2}\Delta_Eh-\delta^{\star_g,-}\delta^gh-\tfrac{1}{2} 
\di^{-} \delta^gh,h \circ H\ra_{L^2}\\
=&\tfrac{1}{2}\la \{h,\tfrac{1}{2}\Delta_Eh-\delta^{\star_g,-}\delta^gh\} ,H \ra_{L^2}-\tfrac{1}{4}\la 
[h,\di^{-} \delta^gh],H\ra_{L^2}.
\end{split}
\end{equation*}
The claim follows now by gathering terms.
\end{proof}
We define the operator $h \in \Gamma \left (\Sym^{2,-}TM \right )\mapsto \mathbf{D}h \in \Gamma \left (\Sym^{2,+}TM \right )$ according to 
$$ \mathbf{D}h:=\{\ring{R}h+\delta^{\star_g,-}\delta^gh,h\}-(\tfrac{1}{2}\Delta_E-\ring{R})h^2+\delta^{\star_g,+}\left ((2\delta^gh^2+\tfrac{1}{2}\di \tr h^2)-2h \delta^gh \right )
+\tfrac{1}{2}[h,\di^{-} \delta^gh].
$$
To render computations easier to follow it is sometimes convenient to isolate the terms containing $\delta^gh$ by splitting 
$\mathbf{D}(h)=\mathbf{D}_1(h)+\mathrm{div}_1(h,h)$ where 
\begin{equation*}
\begin{split}
\mathbf{D}_1(h):=&\{\ring{R}h,h\}-(\tfrac{1}{2}\Delta_E-\ring{R})h^2+\delta^{\star_g,+}\left (2\delta^gh^2+\tfrac{1}{2}\di \tr h^2\right )\\
\mathrm{div}_1(h,h):=&\{\delta^{\star_g,-}\delta^gh,h\}-2\delta^{\star_g,+}(h \delta^gh )
+\tfrac{1}{2}[h,\di^{-} \delta^gh].
\end{split}
\end{equation*}
Based on the preliminary material thus developed we may prove the main result in this section which reads as follows.
\begin{pro} \label{type-I}
Let $h \in \Gamma \left (\Sym^{2,-}TM \right )$. Then 
\begin{equation*}
\begin{split}
\la H \sharp \di_{\nabla^g}h, \di_{\nabla^g}h \ra_{L^2}=&\la \mathbf{D}h ,H\ra_{L^2}
+2\la H \sharp \bdel h, \bdel h \ra_{L^2}+2 \la \delta^g (h \circ \bdel h ), H\ra_{L^2}\\
-&2 \la \bdel h, \delta^gH \wedge h\ra_{L^2}
\end{split}
\end{equation*}
whenever $H$ belongs to $\Gamma \left (\Sym^{2,+}TM \right )$.
\end{pro}
\begin{proof}
We first record that, after expanding in a local orthonormal frame $\{e_i\}$ on $M$, we have 
$$ g(H \sharp \di_{\nabla^g}\!h,\di_{\nabla^g}\!h )=\tfrac{1}{2} g(\di_{\nabla^g}\!h(He_i,e_j)+\di_{\nabla^g}\!h(e_i,He_j),\di_{\nabla^g}\!h(e_i,e_j))=g(\di_{\nabla^g}\!h(e_i,He_j), \di_{\nabla^g}\!h(e_i,e_j)).
$$
According to the definition of $\di_{\nabla^g}$ we compute
\begin{equation} \label{expn-1}
\begin{split}
g(\di_{\nabla^g}\!h(e_i,He_j),\di_{\nabla^g}\!h(e_i,e_j))=&g((\nabla^g_{e_i}h)He_j-(\nabla^g_{He_j}h)e_i,(\nabla^g_{e_i}h)e_j-(\nabla^g_{e_j}h)e_i)\\
=&g(\mathrm{L}(h),H)+g((\nabla^g_{He_j}h)e_i,(\nabla^g_{e_j}h)e_i)\\
-&g((\nabla^g_{He_j}h)e_i,(\nabla^g_{e_i}h)e_j)-g((\nabla^g_{e_i}h)He_j,(\nabla^g_{e_j}h)e_i)\\
=&g(\mathrm{L}(h),H)+g((\nabla^g_{He_j}h)e_i,(\nabla^g_{e_j}h)e_i)\\
-&2g((\nabla^g_{e_i}h)He_j,(\nabla^g_{e_j}h)e_i)
\end{split}
\end{equation}
after also taking into account that $H$ is symmetric. The second summand above may be computed alternatively according to the following observation which essentially takes into account that $HJ=JH$. Indeed, 
\begin{equation*}
\begin{split}
g((\nabla^g_{He_j}h)e_i,(\nabla^g_{e_j}h)e_i)=&g((\nabla^g_{JHe_j}h)Je_i,(\nabla^g_{Je_j}h)Je_i)\\
=&g\left ((\nabla^g_{JHe_j}h)Je_i-(\nabla^g_{He_j}h)e_i,(\nabla^g_{Je_j}h)Je_i \right )\\
+&g\left ((\nabla^g_{He_j}h)e_i,(\nabla^g_{Je_j}h)Je_i \right ).
\end{split}
\end{equation*}
Taking into account the symmetry property in \eqref{sym-h} and the fact that $\bdel h$ is $J$-anti-invariant we compute 
\begin{equation*}
\begin{split}
&g((\nabla^g_{JHe_j}h)Je_i-(\nabla^g_{He_j}h)e_i,(\nabla^g_{Je_j}h)Je_i)\\
=&g((\nabla^g_{Je_i}h)JHe_j-(\nabla^g_{e_i}h)He_j,(\nabla^g_{Je_j}h)Je_i)
+2 g(\bdel h(e_i,He_j),(\nabla^g_{Je_j}h)Je_i)\\
=&g((\nabla^g_{e_i}h)He_j,(\nabla^g_{e_j}h)e_i)-g\left ((\nabla^g_{e_i}h)He_j,(\nabla^g_{Je_j}h)Je_i \right )
-2g(\bdel h(e_i,He_j),(\nabla^g_{e_j}h)e_i)\\
=&g((\nabla^g_{e_i}h)He_j,(\nabla^g_{e_j}h)e_i)+g\left ((\nabla^g_{e_i}h)JHe_j,(\nabla^g_{e_j}h)Je_i \right )
-2g(\bdel h(e_i,He_j),(\nabla^g_{e_j}h)e_i)\\
=&2g((\nabla^g_{e_i}h)He_j,(\nabla^g_{e_j}h)e_i)-2g(\bdel h(e_i,He_j),(\nabla^g_{e_j}h)e_i).
\end{split}
\end{equation*}
To obtain the last two lines above we have operated the basis change $e_j \mapsto Je_j$, used that $HJ=JH$ and also taken into account that covariant derivatives of the type $\nabla^g_Xh$ with $X \in TM$
anti-commute with $J$. It follows that 
\begin{equation*}
\begin{split}
g((\nabla^g_{He_j}h)e_i,(\nabla^g_{e_j}h)e_i)=&2g((\nabla^g_{e_i}h)He_j,(\nabla^g_{e_j}h)e_i)+g\left ((\nabla^g_{He_j}h)e_i,(\nabla^g_{Je_j}h)Je_i \right)\\
-&2g(\bdel h(e_i,He_j),(\nabla^g_{e_j}h)e_i).
\end{split}
\end{equation*}
Replacing this in \eqref{expn-1} shows that 
\begin{equation*}
\begin{split}
g(\di_{\nabla^g}\!h(e_i,He_j),\di_{\nabla^g}\!h(e_i,e_j))=&g(\mathrm{L}(h),H)+g\left ((\nabla^g_{He_j}h)e_i,(\nabla^g_{Je_j}h)Je_i \right)\\
-&2g(\bdel h(e_i,He_j),(\nabla^g_{e_j}h)e_i).
\end{split}
\end{equation*}
Thus gathering terms and taking into account Lemma \ref{pl-4} to express the last summand above leads to 
$$\la H \sharp \di_{\nabla^g}h, \di_{\nabla^g}h \ra_{L^2}= \la 2\ring{R}(h^2)+\mathrm{L}(h),H\ra_{L^2}-\la 
\nabla^g_{\delta^g(HJ)}h,hJ \ra_{L^2}-2g(\bdel h(e_i,He_j),(\nabla^g_{e_j}h)e_i).$$
The claim follows using the expression for $\mathrm{L}(h)$ in Lemma \ref{L-l}, Lemma \ref{div-11} for the divergence type summand in the right-hand side above respectively Lemma \ref{bdd} for the term containing $\bdel h$. 
\end{proof}
\subsection{Proof of Theorem \ref{main-1i}} \label{proof}
Proposition \ref{type-I} allows computing explicitly the quantity $\bfv(h,h)$ with $h \in \mathscr{E}(M,g)$ as indicated below; at the same time we also fully determine the components 
of $\bfv(h,h)$ with respect to the complex type decomposition of symmetric tensors, that is $\Sym^2TM=\Sym^{2,+}TM \oplus \Sym^{2,-}TM$.
\begin{pro} \label{typeI-v}
Let $h \in \Gamma \left (\Sym^{2,-}TM \right )$ satisfy $\bdel h=0$ and $\delta^gh=0$. Then 
\begin{equation*}
Eh^2-\bfv(h,h)+\tfrac{1}{2}\widetilde{\Delta}_Eh^2=-2\delta^g[h,h]^c-\tfrac{1}{2}\delta^{\star_g,-}\di\!\tr h^2.
\end{equation*}
\end{pro}
\begin{proof}
According to \eqref{bfv-1} we have 
\begin{equation}\label{prr-v}
\begin{split}
\la \bfv(h,h),H\ra_{L^2}=&\la 2\delta^g[h,h]^{\FN}-\delta^g \di_{\nabla^g}h^2,H\ra_{L^2}\\
&-\la H \sharp \di_{\nabla^g}h,\di_{\nabla^g}h\ra_{L^2}+\la H, \{h,\delta^g\di_{\nabla^g} h\}\ra_{L^2}
\end{split}
\end{equation}
whenever $H \in \Gamma \left (\Sym^2TM \right )$. In addition, using Proposition \ref{sym-cpx} together with the comparison formula 
\eqref{wz1} shows that the first line above reads 
\begin{equation} \label{pr-vv}
\begin{split}
&\la  2\delta^g[h,h]^{\FN}-\delta^g \di_{\nabla^g}h^2,H\ra_{L^2}\\
=&\la (\Delta_{E}+2\ring{R})h^2+2\delta^g[h,h]^c-2\delta^{\star_g,-}\delta^gh^2+\delta^{\star_g}\delta^g h^2-(\Delta_E+E+\ring{R})h^2,H \ra_{L^2}\\
=&\la (\ring{R}-E)h^2+(\delta^{\star_g,+}-\delta^{\star_g,-})\delta^g h^2+2\delta^g[h,h]^c,H \ra_{L^2}.
\end{split}
\end{equation}
The rest of the proof amounts to computing the remaining terms in the right hand side of \eqref{prr-v}, according to 
$H \in \Gamma \left (\Sym^{2,+}TM \right )$ or $H \in \Gamma \left (\Sym^{2,-}TM \right )$. \\
(i) Assume that $H \in \Gamma \left (\Sym^{2,+}TM \right )$. Then \eqref{pr-vv} entails 
\begin{equation*}
\begin{split}
\la  2\delta^g[h,h]^{\FN}-\delta^g \di_{\nabla^g}h^2,H\ra_{L^2}=\la \left (\ring{R}-E+\delta^{\star_g,+}\delta^g \right )h^2,H \ra_{L^2}.
\end{split}
\end{equation*}
To compute the remaining terms in \eqref{prr-v} we use this time Proposition \ref{type-I} and again \eqref{wz1} to see that 
$\delta^g\di_{\nabla^g} h=(\ring{R}+E)h$. Thus 
\begin{equation*}
\begin{split}
&-\la H \sharp \di_{\nabla^g}h,\di_{\nabla^g}h\ra_{L^2}+\la H, \{h,\delta^g\di_{\nabla^g} h\}\ra_{L^2}\\
=&\la -\{\ring{R}h,h\}+\tfrac{1}{2}(\Delta_E-2\ring{R})h^2+\{h,Eh+\ring{R}h\}-\delta^{\star_g,+}(2\delta^g+\tfrac{1}{2}\di \tr)h^2,H \ra_{L^2}\\
=&\la (\tfrac{1}{2}\Delta_E-\ring{R}+2E)h^2-\delta^{\star_g,+}(2\delta^g+\tfrac{1}{2}\di \tr)h^2,H \ra_{L^2}.
\end{split}
\end{equation*}
Gathering these terms in \eqref{prr-v} whilst using the definition of $\widetilde{\Delta}_E$ shows that 
\begin{equation*}
\begin{split}
\la Eh^2-\bfv(h,h)+\tfrac{1}{2}\widetilde{\Delta}_Eh^2, H\ra_{L^2}=0.
\end{split}
\end{equation*}
(ii) Now assume that $H$ belong to $\Gamma \left (\Sym^{2,-}TM \right )$. Since the  operator $\ring{R}-E$ preserves the space $\Gamma \left (\Sym^{2,+}TM \right )$ we find that 
\eqref{pr-vv} reduces to 
\begin{equation*}
\begin{split}
&\la 2 \delta^g[h,h]^{\FN}-\delta^g \di_{\nabla^g}h^2,H\ra_{L^2}
=\la 2\delta^g[h,h]^c-\delta^{\star_g,-}\delta^gh^2,H \ra_{L^2}.
\end{split}
\end{equation*}
Finally, having $\di\!_{\nabla^g}h \in \Omega^{1,1}(M,TM)$ entails that $H \sharp \di_{\nabla^g}h$ is a section of $\lambda^2(M,TM)$ hence the scalar product $\la H \sharp \di_{\nabla^g}h,\di_{\nabla^g}h\ra_{L^2}=0$. Furthermore, using \eqref{wz1} as previously, we obtain that 
the anti-commutator $\{h,\delta^g\di_{\nabla^g} h\}=\{h,Eh+\ring{R}(h)\}$ belongs to $\Gamma \left ( \Sym^{2,+}TM \right )$ 
hence $\la H, \{h,\delta^g\di_{\nabla^g} h\}\ra_{L^2}=0.$ We have showed that 
$$-\la H \sharp \di_{\nabla^g}h,\di_{\nabla^g}h\ra_{L^2}+\la H, \{h,\delta^g\di_{\nabla^g} h\}\ra_{L^2}=0.$$
Plugging the last two displayed equations in \eqref{prr-v} thus yields 
\begin{equation*} 
\la \bfv(h,h),H\ra_{L^2}=\la 2\delta^g[h,h]^c-\delta^{\star_g,-}\delta^gh^2,H \ra_{L^2}.
\end{equation*}
The claim follows now from $\la \widetilde{\Delta}_Eh^2,H\ra_{L^2}=-\la \delta^{\star_g,-}(2\delta^gh^2+\di\!\tr(h^2)),H\ra_{L^2}$.
\end{proof}
As a consequence of these computations we show below that the Einstein deformation theory to 
order $2$ is entirely explicit. Strikingly, the Hermitian component of $h_2$ is determined algebraically in terms of $h_1$ only.
\begin{teo} \label{main-1}
Let $(M^{2m},g,J)$ be K\"ahler-Einstein with $E<0$ and let $g_t$ be a normalised Einstein deformation of $g$ with $g_0=g$. Then 
up to a suitable gauge transformation $g_t \mapsto f_t^{\star}g_t$ with $f_t \in \mathbf{G}$ we have
$$ h_2^{+}=h_1^2 \ \mathrm{and} \ h_2^{-}=\mathbf{h}_2-\tfrac{1}{2}\L_{J\grad \mathbf{f}}J
$$
where the pair $(\mathbf{h}_2, \mathbf{f})$ in $\TT^{-}(M,g) \oplus C^{\infty}M$ satisfies 
\begin{equation*}
\begin{split}
&\Delta_E^g \mathbf{h}_2=-2\delta^g[h_1,h_1]^c \ \mathrm{and} \  (\Delta^g-2E)\mathbf{f}=-\tfrac{1}{2}\tr(h_1^2).
\end{split}
\end{equation*}
\end{teo}
\begin{proof}
We use the normalisation from Theorem \ref{norm-thm} so that 
\begin{equation*}
\Delta_E(h_2-h_1^2)=\tfrac{1}{2}\widetilde{\Delta}_Eh_1^2+Eh_1^2-\bfv(h_1,h_1).
\end{equation*}
Next take into account that $\Delta_E$ preserves complex type and use the expression for $\bfv(h_1,h_1)$ obtained 
in Proposition \ref{typeI-v}. Thus 
\begin{equation} \label{sys-L}
\begin{split}
&\Delta_E(h_2-h_1^2)^{+}=0\\
&\Delta_Eh_2^{-}=-\tfrac{1}{2}\delta^{\star_g,-}\di\!\tr h_1^2-2\delta^g[h_1,h_1]^c.\\
\end{split}
\end{equation}
Since $E<0$ the comparison formula \eqref{lap-comp2} ensures that $(h_2-h_1^2)^{+}=0$, that is $h_2-h_1^2$ belong to 
$\Sym^{2,-}TM$. To solve the last equation in \eqref{sys-L} we proceed as follows; since $E<0$ the operator $\Delta^g-2E$ acting on functions is invertible. Hence there exists a unique solution $\mathbf{f}$ to $(\Delta^g-2E)\mathbf{f}=-\tfrac{1}{2}\tr(h_1^2)$.
Furthemore, taking into that $\Delta_E$ preserves complex tensor type and Lemma \ref{E-div},(i) we find 
$\Delta_E \delta^{\star_g,-}X =\delta^{\star_g,-}(\Delta^g-2E)X$. Thus letting $X:=\grad \mathbf{f}$ 
we see that the second equation in \eqref{sys-L} amounts to $\Delta_E \mathbf{h}_2=-2\delta^g[h_1,h_1]^c$ where 
$\mathbf{h}_2:=h_2^{-}-\delta^{\star_g,-}X$ belongs to $\Sym^{2,-}TM$. Note that $\delta^{\star_g,-}X=-\tfrac{1}{2}\L_{J \grad \mathbf{f}}J$ as entailed by \eqref{delta-K}.

Since $\delta^g[h_1,h_1]^c$ is divergence free by Proposition \ref{sym-N} it follows that 
$$ 0=\delta^g \Delta_E \mathbf{h}_2=(\Delta^g-2E) \delta^g\mathbf{h}_2
$$
after also taking into account Lemma \ref{E-div},(ii). Because $E<0$ we conclude that the tensor 
$\mathbf{h}_2$ is divergence free, and thus belongs to $\TT^{-}(M,g)$. The claim is now fully proved.
\end{proof}
We end the paper with the following 
\begin{rema}\label{rmk-4.22}
The tensor $\bdel \mathbf{h}_2+[h_1,h_1]^c$ clearly belong to $\ker \bdel$; see also Remark \ref{del-b}. By Corollary \ref{bde} we have 
$\delta^g (\bdel \mathbf{h}_2+[h_1,h_1]^c)=\tfrac{1}{2}\Delta_E\mathbf{h}_2+\delta^g[h_1,h_1]^c=0
$
since $\mathbf{h}_2$ is divergence free. Thus the tensor 
$$\bdel \mathbf{h}_2+[h_1,h_1]^c \in \ker \bdel \cap \ker \delta^g .$$
The vanishing of the harmonic tensor $\bdel \mathbf{h}_2+[h_1,h_1]^c$ characterises the obstruction to second order deformations of the complex structure $J$; see also Remark \ref{del-b}, (i).
\end{rema}


\end{document}